\documentclass[12pt]{article}
\usepackage{cite}
\usepackage{fullpage}
\usepackage{amssymb}
\usepackage{amsthm}
\usepackage{amsmath}
\usepackage{graphicx}
\usepackage[all,2cell]{xy}
\usepackage{listings}
\usepackage{young}
\usepackage{enumitem}
\usepackage[citecolor=black,urlcolor=blue,bookmarks=false,hypertexnames=true]{hyperref} 
\newtheorem{theorem}{Theorem}[section]
\newtheorem{corollary}[theorem]{Corollary}
\newtheorem{definition}[theorem]{Definition}

\newtheorem{lemma}[theorem]{Lemma}
\newtheorem{notation}[theorem]{Notation}
\newtheorem{proposition}[theorem]{Proposition}
\newtheorem{observation}[theorem]{Observation}
\newtheorem{claim}[theorem]{Claim}
\newtheorem{remark}[theorem]{Remark}

\newcommand{\eat}[1]{}

\newcommand{\grass}[2]{G_{{#1},{#2}}}

\newcommand{\taumn}[2]{w_{#1,#2}}
\newcommand{\gmodp}{G/P}
\newcommand{\gmodq}{G/Q}
\newcommand{\gmnmodt}[2]{T \backslash\mkern-6mu\backslash (\grass{#1}{#2})}
\newcommand{\xtaumodt}[2]{T\backslash\mkern-6mu\backslash X(\taumn{#1}{#2})}

\newcommand{\richmodt}[3]{T\backslash\mkern-6mu\backslash (X^{#2}_{#1})^{ss}_T(\
#3)}
\providecommand{\keywords}[1]{\textbf{\textit{Key words---}} #1}
\providecommand{\subject}[1]{\textbf{\textit{Subject Classification---}} #1}
\begin{document}
\title{Torus quotients of Richardson varieties in the Grassmannian}
\author{
Sarjick Bakshi
\thanks{Chennai Mathematical Institute, Chennai, India, {\tt sarjick@cmi.ac.in}}
\and
S. Senthamarai Kannan \thanks{Chennai Mathematical Institute, Chennai,
  India, {\tt kannan@cmi.ac.in}}
\and
K.Venkata Subrahmanyam \thanks{Chennai Mathematical Institute,
  Chennai, India, {\tt kv@cmi.ac.in}}
}
\date{October 25, 2018}
\maketitle
\begin{abstract}
We study the GIT quotient of the minimal Schubert variety in the Grassmannian admitting semistable points for the action of maximal torus $T$, with respect to the $T$-linearized line bundle ${\cal L}(n \omega_r)$ and show that this is smooth when $gcd(r,n)=1$. When $n=7$ and $r=3$ we study the GIT quotients of all Richardson varieties in the minimal Schubert variety. This builds on previous work by Kumar \cite{kumar2008descent}, Kannan and Sardar \cite{kannan2009torusA}, Kannan and Pattanayak \cite{kannan2009torusB}, and recent work of Kannan et al \cite{kannan2018torus}. It is known that the GIT quotient of $G_{2,n}$ is projectively normal. We give a different combinatorial proof.

\end{abstract}

\keywords{GIT,Semistable points, Projective normality, Richardson varieties.}\\

\subject{14M15}

\section{Introduction}\label{s.Introduction}
Let $G$ be a simply connected semi-simple algebraic group over ${\mathbb C}$. Let $T$ be a maximal torus of $G$. Let $B$ be a Borel subgroup of $G$ containing $T$. We denote by $B^{-}$ the Borel subgoup of $G$ opposite to $B$ determined by $T$. Let $Q \supseteq B$ be a parabolic subgroup of $G$ containing $B$. Then $\gmodq$ is a projective variety (see, Jantzen \cite{jantzen2007representations}). Let ${\cal L}$ be a $T$-linearized ample line bundle on $\gmodq$. A point $p \in \gmodq$ is said to be semistable with respect to the $T$-linearized line bundle ${\cal L}$ if there is a $T$-invariant section of a positive power of ${\cal L}$ which does not vanish at $p$.  We denote by $({\gmodq})^{ss}_T({\cal L})$ 
the set of all semistable points with respect to ${\cal L}$. A point in  $({\gmodq})^{ss}_T({\cal L})$ is said to be stable if its $T$-orbit is closed in $(\gmodq^{ss})_T(\cal L)$ and its stabilizer in $T$ is finite.  Let $({\gmodq})^{s}_T({\cal L})$ denote the set of all stable points with respect to ${\cal L}$. This paper is motivated by the question of understanding the GIT quotient of $\gmodq$ with respect to the $T$-linearized bundle ${\cal L}$. 

When $G$ is of type A this problem has been well studied. There is a reasonable body of work when $G$ is of type other than A. Although the results in this paper pertain only to the case when $G$ is of type A we will nevertheless assume that $G$ is of general type  in this introduction to give a comprehensive survey of the results known. For that we will need to introduce some notation. We follow the notation from Lakshmibai and Raghavan,\cite{lakshmibai2007standard}.

Let $X(T)$ denote the group of characters of $T$.  In the root system $R$ of $(G,T)$ let $R^{+}$ denote the set of positive roots with respect to $B$. Let
Let  $S=\{\alpha_1,\ldots,\alpha_l\}\subseteq R^{+}$ denote the set of
simple roots and $\{\omega_1,\ldots,\omega_l\}$ the fundamental weights. Let $U$ (respectively, $U^{-}$) be the unipotent radical of $B$ (respectively, $B^{-}$). For each $\alpha \in R^{+}$,
let $U_{\alpha}$ (respectively, $U^{-}_{-\alpha}$) denote the additive one-dimensional subgroup of $U$ (respectively, $U^{-}$) corresponding to the root $\alpha$ (respectively, $-\alpha$) normalized by $T$. 

Let $N_G(T)$ denote the normalizer of $T$ in $G$. The Weyl group $W$ of $G$ is defined to be the quotient $N_G(T)/T$, and for every $\alpha \in R$ there is a corresponding reflection $s_{\alpha} \in W$.  $W$ is generated by $s_{\alpha}$, $\alpha$ running over simple roots in $S$. This also defines a length function $l$ and the Bruhat order on $W$.

For a subset $I\subseteq S$ denote $W^{I}=\{w\in W|w(\alpha)>0,\alpha\in I\}$ and $W_{I}$ be the subgroup of $W$ generated by $s_{\alpha}$, $\alpha \in I$. 
Then every $w\in W$ can be
uniquely expressed as $w=w^{I}w_{I}$, with $w^{I}\in W^{I}$ and $w_{I}\in W_{I}$. 
For $w\in W$, let $n_w \in N_{G}(T)$ be a representative of $w$.  We denote by $P_{I}$ the parabolic subgroup 
of $G$ generated by $B$ and $n_{w}$, $w \in W_{I}$. Then $W_{I}$ is the Weyl group of the parabolic subgroup $P_{I}$ and abusing notation we also denote it as $W_{P_{I}}$.
When $I = S \setminus \{\alpha_r\}$, has cardinality one less than the cardinality of $S$, we denote the corresponding maximal parabolic subgroup of $G$ by $P_{\hat{\alpha_r}}$. 

The quotient space $\gmodp$ is a homogenous space for the left action of $G$.
The $T$ fixed points in $G/P$ are $e_w=wP/P$ with $w \in W^{P}$. The $B$-orbit $C_w$ of $e_w$, is called a Bruhat cell and it is an affine space of dimension $l(w)$. The closure of $C_w$ in $\gmodp$ is the Schubert variety $X(w)$. The opposite Bruhat cell $C^w$ is the $B^{-}$ orbit of $e_w$ and its closure, denoted by $X^w$, is the opposite Schubert variety. For a $T$-linearized line bundle ${\cal L}$ on a Schubert variety in $G/P$ we define the notion of semistable and stable points as before. We use the notation $X(w)^{ss}_T({\cal L})$ (respectively,  $X(w)^{s}_T({\cal L})$) to denote the semistable (respectively, stable) points for the $T$-linearized line bundle ${\cal L}$.

Every character $\lambda$ of $P$ defines a $G$-linearized line bundle on $G/P$. We denote the line bundle by ${\cal L}(\lambda)$. Furthermore, ${\cal L}(\lambda)$ is generated by global sections if and only if $\lambda$ is a dominant weight,(see \cite[Part II, Proposition 2.6]{jantzen2007representations}). 

When $G = SL(n,\mathbb{C})$ and $P=P_{\hat{\alpha_r}}$, $\gmodp$ is the Grassmannian parametrizing $r$-dimensional subspaces of  $\mathbb{C}^n$. We denote it by $\grass{r}{n}$. The Grassmannian $\grass{r}{n}$ comes with the Plucker embedding 
 $\grass{r}{n} \hookrightarrow \mathbb{P}(\bigwedge^r\mathbb{C}^n)$
 sending each $r$-dimensional subspace to its $r$-th exterior wedge
 product (see Fulton \cite{fulton1997young}).  The pull back of ${\cal O}(1)$ from the projective space to $\grass{r}{n}$ is an ample generator of the Picard group of 
$\grass{r}{n}$ and corresponds to the $T$-linearized line bundle ${\cal L}(\omega_r)$.
Gel'fand and Macpherson \cite{gelfand1982geometry}, considered the GIT quotient of the Grassmannian and  showed that the
$GIT$ quotient of 
$n$-points in ${\mathbb P}^{r-1}$ (spanning ${\mathbb P}^{r-1}$) by the diagonal action of $PGL(r,\mathbb{C})$ is isomorphic to the GIT quotient of $\grass{r}{n}$ with respect to the $T$-linearized line bundle ${\cal L}(n \omega_r)$.
They showed that the torus action gives rise to a moment map from $\grass{r}{n}$ to   ${\mathbb R}^n$, with the property that the image of each orbit is a convex polyhedron.
This was extended by Gelfand et al in \cite{gelfand1987combinatorial}. In loc.cit. the authors proposed  three natural ways to stratify the Grassmannian - the first stratification is motivated by the equivalence of the torus quotient with the configuration of points in 
${\mathbb P}^{k-1}$, the second is motivated by the moment map above, and the third is motivated by the geometry of intersections of Schubert cells in the Grassmannian.
The authors show that no matter which definition is used to stratify the Grassmannian, the strata are the same.

Hausmann and Knutson \cite{hausmann1997polygon} used the GGMS stratification to study the GIT quotient of $\grass{2}{n}$
and related the resulting GIT quotient to the moduli space of polygons in ${\mathbb R}^3$.

Using the Hilbert-Mumford criterion, Skorobogatov \cite{skorobogatov1993swinnerton} gave combinatorial conditions determining when a point in $\grass{r}{n}$ is semistable with respect to the $T$-linearized bundle ${\cal L}(\omega_r)$. As a corollary he showed that when $r$ and $n$ are coprime semistability is the same as stability.  

Independently, for a general $G$, Kannan \cite{kannan1998torus} and \cite{kannan1999torus}, gave a description of parabolic subgroups $Q$ of $G$ for which there exists an ample line bundle $\mathcal{L}$ on $G/Q$ such that $(G/Q)^{ss}_{T}(\mathcal{L})$ is the same as 
$(G/Q)^{s}_{T}(\mathcal{L})$. In particular, in the case when $G=SL(n,\mathbb{C})$ and $Q=P_{\hat{\alpha_r}}$, Kannan showed that $(\grass{r}{n})^{s}_T({\cal L}(\omega_r))$
is the same as $(\grass{r}{n})^{ss}_T({\cal L}(\omega_r)$ if and only if $r$ and $n$ are coprime.

In the type $A$ case when $G=SL(n,\mathbb{C})$ and $Q$ is a parabolic subgroup, Howard \cite{howard2005matroids} considered the problem of determining which line bundles on $G/Q$ descend to ample line bundles of the GIT quotient of $G/Q$ by $T$. For a line bundle which descends to an ample line bundle on the quotient, by the Gelfand-MacPherson correspondence, the smallest power of
the descent bundle that is very ample would give an upper bound on the degree 
in which the ring of invariants of $n$-points spanning projective space ${\mathbb P}^{r-1}$ is generated. Howard showed that when ${\cal L}(\lambda)$ is a very ample line bundle on $G/Q$ (so the character of $T$ extends to $Q$ and to no larger subgroup of $G$) and 
$H^0(G/Q, {\cal L}(\lambda))^T$ is non-zero, the line bundle descends to the quotient\cite[Proposition 2.3, Theorem 2.3]{howard2005matroids}. He extended these results to the case when the $T$-linearization of ${\cal L}(\lambda)$ is twisted by
$\mu$, a character of $T$. He proved that the line bundle ${\cal L}(\lambda)$ twisted by $\mu$ descends to the GIT quotient provided the $\mu$-weight space of
$H^0(G/Q, {\cal L}(\lambda))$ is non-zero and this is so when $\lambda - \mu$ is in the root lattice and $\mu$ is in the convex hull of the Weyl orbit of $\lambda$.
This was extended to other algebraic groups by Kumar\cite[Theorem 3.10]{kumar2008descent}.

Kannan and Sardar \cite{kannan2009torusA} studied torus quotients of Schubert varieties in $\grass{r}{n}$. They showed that $\grass{r}{n}$ has a unique minimal Schubert variety,
$X(w_{r,n})$  admitting semistable points with respect to the $T$-linearized 
bundle ${\cal L}(\omega_r)$, and gave a
combinatorial characterization of  $w_{r,n}$.

Kannan and Pattanayak, \cite{kannan2009torusB} extended the results of \cite{kannan2009torusA} to the case when $G$ is of type $B,C$ or $D$ and when $P$ is a maximal parabolic subgroup of $G$. 
Then $G/P_{\hat{\alpha_r} }$ has an ample line bundle ${\cal L}(\omega_r)$.
Kannan and Pattanayak gave a combinatorial description of all minimal Schubert varieties in $G/B$ admitting semistable points with respect to ${\cal L}(\lambda)$ for any dominant character $\lambda$ of $B$.
  
Kannan et al \cite{kannan2018torus} extended the results in \cite{kannan2009torusA} to Richardson varieties in the Grassmannian $\grass{r}{n}$. Recall that a Richardson variety  in $\grass{r}{n}$ is the intersection of the Schubert variety $X(w)$ in $\grass{r}{n}$ with the opposite Schubert variety $X^{v}$ in $\grass{r}{n}$.  In \cite{kannan2018torus} the authors gave a criterion for Richardson varieties in $\grass{r}{n}$ to admit semistable points with respect to the $T$-linearized line bundle $\mathcal{L}(\omega_r)$. 

\subsection{Our results and Organization of the paper}
For all the results in this paper we assume $G$ is of type A.  In Section \ref{s:taurngen} we begin a study of the GIT quotient of
$\grass{r}{n}$ when $r$ is bigger than 2 and $(r,n)=1$. 
We study the GIT quotient of the minimal Schubert variety $X(w_{r,n})$ having semistable points with respect to the $T$-linearized line bundle ${\cal L}(n\omega_r)$. We show that $\xtaumodt{r}{n}^{ss}_T({\cal L}(n \omega_r))$is smooth. We show 
that $w_{r,n} = cv_{r,n}$, where $c$ is a Coxeter element  (i.e each simple reflection occurs exactly once in a reduced expression for $c$) and $l(w_{r,n})= n-1+ l(v_{r,n})$ with $v_{r,n}$ being the (unique) maximal element $v_{r,n} \in W^{S \setminus \alpha_r}$ such that $v_{r,n}(n \omega_r) \geq 0$.  We show that the GIT quotient $\richmodt{w_{r,n}}{v_{r,n}}{{\cal L}(n \omega_r)}$ 
 is a point.  We prove that the GIT quotient $\richmodt{w_{r,n}}{u}{{\cal L}(n \omega_r)}$ 
 is ${\mathbb P}^1$ precisely when one can write $v_{r,n}=s_{\alpha}u$, with $l(v_{r,n})=l(u)+1$ and $\alpha$ is a simple root. We determine all such simple roots and give a description of the descent line bundle to ${\mathbb P}^1$, in terms of the combinatorics of $r,n$.

In Section \ref{s:tau37} we show that the polarized variety $(\xtaumodt{3}{7}^{ss}_{T}({\cal L}(7 \omega_3)), \tilde{\cal L}(7 \omega_3))$ is projectively normal. In Section \ref{s:deodhar} we explicitly calculate the GIT quotients of  Richardson strata in $X(w_{3,7})$ with respect to the $T$-linearized line bundle ${\cal L}(7 \omega_3)$.
We show that $(\xtaumodt{3}{7}^{ss}_{T}({\cal L}(7\omega_3)), \tilde{{\cal L}}(7 \omega_3))$ is a rational normal scroll.  Finally in Section \ref{s:pnormalg2n} prove that when $n$ is odd, the polarized variety $(\gmnmodt{2}{n}_T^{ss}({\cal L}(n\omega_2)), \tilde{{\cal L}}(n \omega_2))$ is projectively normal, a result that is well known (see, \cite{howard2005matroids, Howe}). However we believe that the combinatorics we develop to reprove this result may be useful to extend this result to Grassmannians of higher ranks.\footnote{In personal communication Pattanayak informs us that he and Arpita Nayek have a proof of projective normality of the GIT quotient of $G_{2,n}$ and they have a counter example for Grassmannians of higher ranks.}

{\bf Acknowledgements}: 
S.B was supported by a research fellowship from the National Board of Higher Mathematics. 
All three authors were partially supported by a grant from the Infosys foundation. 
The third author was supported by a grant under the MATRICS scheme of the DST.

\eat{\subsection{Organization}

In the next section we first revisit the notation from the introduction, giving details of the terms and notations used in the rest of paper, with emphasis on the type A case. 
In section
\ref{s:taurngen} we state and prove results of a general nature which
hold for the GIT quotient of $X(w_{r,n})$ with respect to the
$T$-linearized line bundle ${\cal L}(n \omega_r)$. In section \ref{s:tau37} we focus on the $r=3, n=7$ case. We show that the polarized variety $(\xtaumodt{3}{7}^{ss}_{T}({\cal L}(7\omega_3)), \tilde{{\cal L}}(7 \omega_3))$ is projectively normal. We recall the Deodhar decomposition and Richardson strata in the Deodhar decomposition, and explicitly calculate the GIT quotients of Richardson strata in $X(w_{3,7})$ with respect to ${\cal L}(7 \omega_3)$.
}

\section{Notations and Preliminaries}
\label{s.notation}
For the rest of this paper we will assume that $G=SL(n,\mathbb{C})$ and $P$ is a maximal subgroup of $G$.
Keeping this in mind, we revisit the notation developed in the previous section for $G$ of arbitrary type - we derive formulas for the various terms introduced and explicitly write down elements of the Weyl group and the action of the torus.

We take $T$ to be the group of diagonal matrices in $G$, and $B$ the 
subgroup of upper triangular matrices in $G$ and $B^{-}$ the subgroup of lower triangular matrices in $G$. The unipotent subgroup $U$ is the subgroup of $B$ with diagonal entries 1, and $U^{-}$ is the unipotent subgroup of $B^{-}$ with diagonal entries 1. $S = \{\alpha_1,\ldots, \alpha_{n-1}\}$ is the set of simple roots where $\alpha_i = \epsilon_i - \epsilon_{i+1}$, see \cite[Chapter 3]{lakshmibai2007standard}. The Weyl group of $G$ is the permutation group $S_n$ and is generated by the simple
reflections $s_{\alpha_1}, \ldots, s_{\alpha_{n-1}}$ which for simplicity we denote as $s_1,\ldots,s_{n-1}$.

Let $\{e_1,\ldots,e_n\}$ be the standard basis of $\mathbb{C}^n$. Note that for $r \in \{2,\ldots, n\}$, $P_{\hat{\alpha_r}} = \begin{bmatrix}
          * & * \\
	    0_{n-r,r}  &  * 
          \end{bmatrix}$ is the stabilizer of $<e_1,e_2,\ldots,e_r>$ in $G$. $G/P_{\hat{\alpha_r}}$ is the Grassmannian $\grass{r}{n}$, of $r$-dimensional subspaces of
$\mathbb{C}^n$ and this carries a transitive action of $G$ making it a homogenous $G$-variety. $W_{P_{\hat{\alpha_r}}}$ is the subgroup of $W$ generated by simple reflections $s_{\alpha}, \alpha \in S\backslash\{\alpha_r\}$.

Let $I(r,n) = \{(i_1,i_2,..i_r) | 1 \leq i_1 < i_2 .. <i_r \leq n \}$. Then there is a natural identification of  $W^{S\backslash\{\alpha_r\}}$ with $I(r,n)$ sending $w \in W^{S\backslash\{\alpha_r\}}$ to $(w(1),w(2),\ldots,w(r))$.  For $w$ in $I(r,n)$, let $e_{w}=[e_{w(1)} \wedge e_{w(2)} \cdots e_{w(r)}] \in {\mathbb P}(\bigwedge^r\mathbb{C}^n)$. Then $e_w$ is a $T$-fixed point of $\grass{r}{n}$ and it is known that $e_w, w\in I(r,n)$  are precisely the $T$-fixed points of  $\grass{r}{n}$. 
The $B$-orbit through $e_w$ is the Schubert cell and its Zariski closure in $G/P$ is the Schubert variety $X(w)$.  The Schubert cell is the $U$-orbit of $e_w$. The Bruhat order is the order on $r$-tuples in $I(r,n)$ given by containment of Schubert varieties - in this order $v \leq w$ iff $v(i) \leq w(i)$ for $1 \leq i \leq r$.  As mentioned in the previous section $\grass{r}{n}$ comes with a natural line bundle  ${\cal L}(\omega_r)$ and a Pl\"{u}cker embedding in 
$\mathbb{P}(\bigwedge^r\mathbb{C}^n)$. 
\footnote{This notation, valid for type A,  is consistent with the notation set up in the introduction for all types.}
 
Let $\mathbb{C}[X(w)]$ be the homogeneous coordinate ring of $X(w)$ for this projective embedding. From the main theorem of 
standard monomial theory for $SL_n$ \cite[Chapter 4]{lakshmibai2007standard} we get $H^0(X(w),\mathcal{L}(d\omega_r)) = \mathbb{C}[X(w)]_d$, and this has a basis
consisting of $T$-eigenvectors $p_{\tau_1}p_{\tau_2} \cdots p_{\tau_d}$, with $\tau_1\leq\tau_2\leq \ldots \leq\tau_d \leq w$. Here $\tau_i \in I(r,n)$ and the order is the Bruhat order. The weight of $p_{\tau_1} \cdots p_{\tau_d}$ is
$-\sum_i\tau_i(\omega_r)$. 

We associate with each standard monomial $p_{\underline{\tau}}=p_{\tau_1} \cdots p_{\tau_d}$ a semistandard Young tableau $T_{\underline{\tau}}$ of shape $\underbrace{(d,d,..,d)}_\text{$r$ times}$ whose $i$-th column is filled with $\tau_i=[\tau_{i}(1),\tau_{i}(2),\ldots,\tau_{i}(r)]$, see \cite[Chapter 1]{seshadriintroduction}.  It is clear that the rows of the semistandard Young tableau are weakly increasing and the columns are strictly increasing. Let $a(i)$ denote the 
number of times integer $i$ appears in the tableau. Then we have $ diag(t_1,..t_n).p_{\tau_1}p_{\tau_2}\cdots p_{\tau_d} = \prod_i t_i^{a(i)} p_{\tau_1} \cdots p_{\tau_d}$. Since $t_1 \cdots t_n = 1$, a standard monomial
is a zero-weight vector iff all $a(i)$'s appear the same number of times in the Young tableau. \label{tabconv}

First recall that given $(b_1,\ldots,b_r) \in I(r,n)$, one reduced expression for the Weyl group element in $W^{S\backslash\{\alpha_r\}}$ corresponding to this is $(s_{b_1 -1 } \cdots s_1) \ldots (s_{b_r -1} \cdots s_r)$ where a bracket is assumed to be empty is if $b_i -1$ is less than $i$.

We recall some lemmas and propositions which have appeared earlier. We state them nevertheless since this will be required in the rest of the paper. Some of these are folklore.

The following lemma appears in \cite{kumar2008descent}, \cite{kannan1998torus}.

 \begin{lemma}
\label{lm:zwt}
Let $r$ and $n$ be coprime. Let $v \neq 0$ be a zero-weight vector in $H^0(X(w),\mathcal{L}(\omega_r)^{\otimes d})$. Then $n$ divides $d$.
\end{lemma} 
\begin{proof}
Since $0$ is a weight, $d \omega_r$ is in the root lattice. So $n$ divides $d$. 
\end{proof}

Recall from \cite{kannan2009torusA}, that there is a unique minimal Schubert Variety $X(w_{r,n})$ in $G_{r,n}$ admitting semistable points with respect to the line bundle $ \mathcal{L}(n\omega_r)$. 
For completeness we explicitly calculate $w_{r,n}$.

\begin{proposition}
\label{wrn} Let $r$ and $n$ be coprime. Then $w_{r,n} = (a_1,a_2,..a_r)$ where $a_i$ is the smallest integer such that $a_i \cdot r \geq i \cdot n$.
\end{proposition}
\begin{proof}	

Clearly $ w_{r,n} > id$ since $X(id)$ is a point. Let $\alpha$ be a simple root with $s_{\alpha} w_{r,n} \leq w_{r,n}$. Note, $s_{\alpha} w_{r,n} \in W^{S\backslash\{\alpha_r\}}$.
We have a surjection 
$H^0(X(w_{r,n}),\mathcal{L}(n\omega_r))\rightarrow H^0(X(s_{\alpha} w_{r,n}),\mathcal{L}(n\omega_r))$. Let $K$ denote its kernel. So we have 
a short exact sequence

$0\rightarrow K\rightarrow H^0(X(w_{r,n}),\mathcal{L}(n\omega_r))\rightarrow H^0(X(s_{\alpha} w_{r,n}),\mathcal{L}(n\omega_r))\rightarrow 0$.

From minimality of $w_{r,n}$ we get  $K^T\rightarrow H^0(X(w_{r,n}),\mathcal{L}(n\omega_r))^T$ an isomorphism. Now if we choose a 
standard monomial $p_{\underline{\tau}}=p_{\tau_{1}} \cdots p_{\tau_{r}}$ in  $H^0(X(w_{r,n}),\mathcal{L}(n\omega_r))^T$ then we have $\tau_{r} = w_{r,n}$. 

To get such a monomial, we need a filling of the associated tableau $T_{\underline{\tau}}$ with $rn$ boxes such that each $k$, 
$1 \leq k \leq n$ appears exactly $r$-times and the last column is as small as possible in the Bruhat order. \eat{Let $a_{i,j}$ be the entry in the $i$-th row and $j$-th column. We reindex the 
$rn$ boxes as $z_{(i-1)n+j} = a_{i,j}$.} Clearly the filling which results in the smallest element in the Bruhat order appearing as the last column is the one in which the tableau is filled from left to right and top to bottom with numbers $1,2,\ldots,n$, in order, with each appearing exactly $r$ times - so the first entry of the last column is the least integer $a_1$ such that $r a_1 \geq n$ and, in general, the $k$-th entry in the last column is the smallest integer $a_k$ such that 
$r \cdot a_k \geq kn$, completing the proof.
\end{proof}

The tableau constructed in the proof of Proposition \ref{wrn} will be used repeatedly in the paper. We denote it by $\Gamma_{r,n}$. The figure below gives $\Gamma_{3,8}$.

\[
\Gamma_{3,8} = \begin{tabular}{|c|c|c|c|c|c|c|c|}
\hline
1 &1	& 1&	2&	2&	2&	3 & 3\\ 
\hline
3&	4	&4	&4	&5	&5&	5 & 6\\   
\hline                
6&	6&	7&	7&	7&	8&	8 & 8 \\ 
\hline
\end{tabular} 
\]

\section{GIT quotients of Richardson varieties in $X(w_{r,n})$}
\label{s:taurngen}
\eat{In this section we prove results of a general nature which apply to the GIT quotient
of $X(w_{r,n})$ with respect to the linearised bundle ${\cal L}(n \omega_r)$.} The results in this section pertain to GIT quotients of Richardson varieties in $G/P_{\hat{\alpha_r}}$ with respect to the $T$-linearised line bundle  ${\cal L}(n \omega_r)$.

\subsection{GIT quotients of Richardson varieties}
\label{ss:gitrich}
We first prove
\begin{theorem}
Let $r$ and $n$ be coprime. Then the GIT quotient $\xtaumodt{r}{n}^{ss}_{T}({\cal L}(n\omega_r))$ is smooth.
\end{theorem}
\begin{proof}
$X(w_{r,n})$ is the minimal Schubert Variety admitting semistable points with respect to ${\cal L}(n\omega_r)$. So $X(w_{r,n})_{T}^{ss}(\mathcal{L}(n\omega_r)) \cap Bw P_{\hat{\alpha_r}}/P_{\hat{\alpha_r}} =\phi $ for all $w < w_{r,n}$. Hence $X( w_{r,n})_{T}^{ss}(\mathcal{L}(n\omega_r)) \subseteq Bw_{r,n} P_{\hat{\alpha_r}}/P_{\hat{\alpha_r}}$.
 Thus $X(w_{r,n})_{T}^{ss}(\mathcal{L}(n\omega_r))$ is a smooth open subset of $X(w_{r,n})$. Since $r$ and $n$ are coprime we have  $X(w_{r,n})_{T}^{ss}(\mathcal{L}(n\omega_r)) = X(w_{r,n})_{T}^{s}(\mathcal{L}(n\omega_r))$, \cite{kannan2014git}.
Let $G_{ad}=G/Z(G)$ be the adjoint group of $G$. Let $\pi: G \rightarrow G_{ad}$ be the natural homomorphism and $T_{ad} = \pi(T)$. 
Note that $\mathcal{L}(n\omega_r)$ is also $T_{ad}$-linearised. Therefore 
 $X(w_{r,n})_{T_{ad}}^{ss}(\mathcal{L}(n\omega_r)) = X(w_{r,n})_{T}^{ss}(\mathcal{L}(n\omega_r))=
 X(w_{r,n})_{T}^{s}(\mathcal{L}(n\omega_r)) = X(w_{r,n})_{T_{ad}}^{s}(\mathcal{L}(n\omega_r))$. Hence for any point $x \in X(w_{r,n})_{T}^{ss}(\mathcal{L}(n\omega_r))$
 the orbit $T_{ad}.x$ is closed in $X(w_{r,n})_{T}^{ss}(\mathcal{L}(n\omega_r))$ and the stabiliser of $x$ is finite. By \cite[Lemma 3.2]{kannan2014git}
 and the proof of example 3.3, loc.cit., the stabiliser of every point of $X(w_{r,n})_{T}^{ss}(\mathcal{L}(n\omega_r))$ in $T_{ad}$ 
 is trivial. Therefore the GIT quotient $\xtaumodt{r}{n}_T^{ss}(\mathcal{L}(n\omega_r))$ is a geometric quotient. Since 
 $X(w_{r,n})_{T}^{ss}(\mathcal{L}(n\omega_r))$ is smooth, $\xtaumodt{r}{n}_T^{ss}(\mathcal{L}(n\omega_r))$ is also smooth. 
 \end{proof}

Recall that a Richardson variety $X^v_w$  in $\grass{r}{n}$ is the intersection of the Schubert variety $X(w)$ in $\grass{r}{n}$ with the opposite Schubert variety $X^{v}$ in $\grass{r}{n}$.

In \cite[Proposition 3.1]{kannan2018torus} the authors give a characterisation of the smallest Richardson variety in $G_{r,n}$ admitting semistable points. From the proof of 
Proposition \ref{wrn} we obtain
\begin{proposition}
Let $r$ and $n$ be coprime. Let $v_{r,n}$ be such that $X_{w_{r,n}}^{v_{r,n}}$ is the smallest Richardson variety in $X(w_{r,n})$ admitting semistable points.  Then $v_{r,n}=[1, a_1,\ldots, a_{r-1}]$ with the $a_i$ defined as 
the smallest integer satisfying $a_i r \geq i\cdot n$ (as in Proposition \ref{wrn}).
\end{proposition}
\begin{proof}
Let $v_{r,n}=[b_1,\ldots, b_r]$. Since $X_{w_{r,n}}^{v_{r,n}}$ has a semistable point, $H^0(X^{v_{r,n}}_{w_{r,n}},\mathcal{L}(n\omega_r))^T$ is non zero. Now  $H^0(X^{v_{r,n}}_{w_{r,n}},\mathcal{L}(n\omega_r))$ has a standard monomial basis $p_{\tau_1} \ldots p_{\tau_n}$ with $\tau_1 \leq \tau_2 \cdots \leq \tau_n$ (see \cite {brion2003geometric}).  We identify this basis with semistandard Young tableau having columns $\tau_1,\tau_2,\ldots,\tau_n$ as before. It follows from this identification that there is a semistandard Young tableau with $r$ rows and $n$-columns in which each integer $1 \leq k \leq n$ appears exactly $r$ times. From Proposition \ref{wrn} and \cite[Proposition 6]{brion2003geometric} and we have $\tau_n = w_{r,n}$ and $v_{r,n} \leq \tau_1$. Since every semistandard Young tableau has each integer in $\{1,\ldots,n\}$ appearing $r$ times and the first entry of $\tau_1$ is always 1, $b_1$ must be 1. Since $r,n$ are coprime, from the definition of $a_1$ it is immediate  that all $a_1$'s can be in the first row. For the same reason the $a_i$'s cannot all appear in the first $i$ rows.  So $a_i$ must appears in a row $j$ where $j > i$. Hence $b_i \leq a_{i-1}$. Note that the first column of the Young tableau $\Gamma_{r,n}$ from Proposition \ref{wrn} is $v=[1,a_1,\ldots,a_{r-1}]$. From \cite[Proposition 6]{brion2003geometric}, the $T$-invariant $\Gamma_{r,n}$ is non zero on $X^v_{w_{r,n}}$. Hence  $v_{r,n}=v=[1,a_1,\ldots, a_{r-1}]$.

\end{proof}
 
Consider the Weyl group element $c_{r,n} = w_{r,n} v_{r,n}^{-1}$ .
We claim
\begin{claim} $c_{r,n}$ is a Coxeter element.
\end{claim}

\begin{proof}

We have a reduced expression $w_{r,n} = (s_{a_1-1} \cdots s_1)(s_{a_2-1} \cdots s_2)\cdots (s_{a_r-1} \cdots s_r)$ and  $v_{r,n} = (s_{a_1-1} \cdots s_2)(s_{a_2-1} \cdots s_3) \cdots (s_{a_{r-1}}-1 \cdots s_r)$ . Then 
\[w_{r,n} v_{r,n}^{-1} = (s_{a_1-1} \cdots s_1)(s_{a_2-1} \cdots s_{a_1}) (s_{a_3-1} \cdots s_{a_2}) \cdots (s_{a_{r}-1} \cdots s_{a_{r-1}})\]
This is a Coxeter element.

\end{proof}

We now consider the GIT quotients of Richardson varieties in $X_{w_{r,n}}$. We first show

\begin{theorem}
\label{thm:pt}
$\richmodt{w_{r,n}}{v_{r,n}}{\mathcal{L}(n\omega_r)}$
is a point.
\end{theorem}
\begin{proof}
Since $dim X_{w_{r,n}}^{v_{r,n}} = l(w_{r,n}) - l(v_{r,n}) = l(c_{r,n}) = n-1 = dim T$ and   $(X_{w_{r,n}}^{v_{r,n}})_T^{ss}(\mathcal{L}(n\omega_r)) = (X_{w_{r,n}}^{v_{r,n}})_T^{s}(\mathcal{L}(n\omega_r))$, so the dimension of the GIT quotient is $0$. Since
$\richmodt{w_{r,n}}{v_{r,n}}{\mathcal{L}(n\omega_r)}$
 is irreducible, $\richmodt{w_{r,n}}{v_{r,n}}{\mathcal{L}(n\omega_r)}$ is a point.  Alternatively, there is a unique standard monomial $p_{\tau_1} p_{\tau_2} \ldots p_{\tau_n}$ with $\tau_1 =[1,a_1,a_2,\ldots, a_{r-1}]$ and 
$\tau_n =[a_1,a_2,\ldots, a_r]$ (the corresponding Young tableau being $\Gamma_{r,n}$).
\end{proof}

\begin{theorem}
Let $v \in W^{S\setminus\{\alpha_r\} }$ be such that $v < v_{r,n}$ . Then, $\richmodt{w_{r,n}}{v}{\mathcal{L}(n\omega_r)}$ is isomorphic to $\mathbb{P}^1$ if and only if  $v = s_{\alpha}v_{r,n}$ where $s_{\alpha} = (a_i-1,a_i)$ for some $i = 1,2, \ldots , r-1.$ The descent of $\mathcal{L}(n\omega_r)$ to $\richmodt{w_{r,n}}{v}{\mathcal{L}(n\omega_r)}$ is $\mathcal{O}_{\mathbb{P}^1}(n_i)$ where $n_i$ is the number of times 
$a_i-1$ appears in the $i$-th row of the tableau $\Gamma_{r,n}$.
\end{theorem}

\begin{proof}
We start with the only if part. Since $X_{w_{r,n}}^{v}$ is normal,  $(X_{w_{r,n}}^{v})^{ss}_T(\mathcal{L}(n\omega_r)) $ is normal and hence $\richmodt{w_{r,n}}{v}{\mathcal{L}(n\omega_r)}$ is normal. Since $dim(\richmodt{w_{r,n}}{v}{\mathcal{L}(n\omega_r)}) = 1$, the GIT quotient  $\richmodt{w_{r,n}}{v}{\mathcal{L}(n\omega_r)}$
 is a smooth, rational projective curve. Hence 
$\richmodt{w_{r,n}}{v}{\mathcal{L}(n\omega_r)}$ is isomorphic to $\mathbb{P}^1$. 

If $\richmodt{w_{r,n}}{v}{\mathcal{L}(n\omega_r)}$ is isomorphic to $\mathbb{P}^1$ we get $l(v) = l(v_{r,n})-1$.  Also $v < v_{r,n}$ and $v \in W^{S\setminus\{\alpha_r\} }$. So $v =(s_{b_i} \cdots s_i) \cdots (s_{b_r} \cdots s_r)$ for some $i$, $1\leq i \leq r$, and for some $1 \leq b_i < b_{i+1} \cdots < b_{r} \leq n-1$ (see the discussion preceding Lemma \ref{lm:zwt}). Since $v_{r,n} =  (s_{a_1-1} \cdots s_2)(s_{a_2-1} \cdots s_3) \cdots (s_{a_{r-1}}-1 \cdots s_r)$, $v = s_\alpha v_{r,n}$ only when 
$s_{\alpha} = (a_i-1,a_i)$, $ 1 \leq i \leq r-1$.

We start with a zero-weight standard monomial basis for $H^0(X^{v}_{w_{r,n}},\mathcal{L}(n\omega_r))$ . Let  $v = s_{\alpha}v_{r,n}$ with $s_{\alpha} = (a_i-1,a_i)$ for some fixed $i$. We have $v_{r,n} = (1, a_1,\ldots,a_i,\ldots,a_{r-1})$ and $v= (1, a_1,\ldots,a_i-1,\ldots, a_{r-1})$. The $i+1$-st entry of $v_{r,n}$ is  $a_i$ and that of $v$ is $a_i-1$ and the rest of the entries are equal. We need to count the number of semistandard tableau of shape $n,n,\ldots,n$ ($r$ rows) with first column $v$. Because the tableau is semistandard, the positions of all integers other than $a_i-1$ and $a_i$ are fixed. So the number of such tableaus depends only on the number of $a_i-1$ in the $i$-th row. $a_i-1$ appears $n_i$ times in the $i$-th row of $\Gamma_{r,n}$. It is easy to see that for every $j$ in $\{0,\cdots,n_i\}$ there is a semistandard tableau with $a_i-1$ appearing $j$ times and $a_i$ appearing $n_i-j$ times in row $i$.
So we have $n_i+1$ linearly independent sections of the descent line bundle on the GIT quotient. This completes the proof.

\end{proof}

\section{Projective normality of the GIT quotient of  $X(w_{3,7})$}.
\label{s:tau37}

In this section we will work with $G=SL(7)$. We use the same notation as before. 
We study the GIT quotient of the Schubert variety $X(w_{3,7})$ with respect to
$T$-linearised line bundle ${\cal L}(7 \omega_3)$.  From \cite[Theorem 3.10]{kumar2008descent} we know that this line bundle descends to the line bundle $\tilde{{\cal L}}(7 \omega_3)$ on the GIT quotient
$\xtaumodt{3}{7}^{ss}_T({\cal L}(7 \omega_3))$. We show 

\begin{theorem}
\label{thm:projnormal}
The polarized variety $(\xtaumodt{3}{7}^{ss}_T({\cal L}(7 \omega_3)), \tilde{{\cal L}}(7 \omega_3))$ is
projectively normal.
\end{theorem}

\begin{remark} 
\label{rem:projnormal}Let $S(m) = H^0(X(w_{3,7}),\mathcal{L}(7\omega_3)^{\otimes m}) $ be the global sections of the line bundle ${\cal L}(7\omega_3)$ on $X(w_{3,7})$ and let $R(m) = H^0(X(w_{3,7}),\mathcal{L}(7\omega_3)^{\otimes m})^T$ denote the invariant subspace with respect to action of $T$. The GIT quotient is precisely $Proj(\oplus_m R(m))$ (see \cite[Proposition 8.1]{dolgachev2003lectures}). 
Since the polarized variety $(X(w_{3,7}), {\cal L}(7\omega_3))$ is projectively normal, we have a surjective map  $S(1)^{\otimes m} \longrightarrow S(m)$ (see \cite{lakshmibai2007standard}) and an induced map $R(1)^{\otimes m} \rightarrow R(m)$.  Now the GIT quotient is smooth it is normal. Therefore to show projective normality of the GIT quotient all we need to show is that $\phi$ is surjective.  
\eat{Since the GIT quotient is smooth it is normal. Hence it suffices to show that $R$ is generated in degree 1, i.e. the multiplication map $R(1)^{\otimes m} \rightarrow R(m)$ is surjective.  }
\end{remark}

From Lemma \ref{wrn} we get $w_{3,7}=[3, 5, 7]$.
As before we identify the standard monomial basis of $H^0(X(w_{3,7}), {\cal L}(7m\omega_3))^T$ with semistandard Young tableaus. These tableaus have 3 rows and $7m$ columns with each integer from $\{1,\ldots,7\}$ appearing exactly $3m$ times - furthermore the last column is $[3,5,7]$.

To aid in the proof of projective normality we list the semistandard Young tableau basis of $R(1)$ and we also write down a semistandard tableau of shape $[14, 14,14]$ from $R(2)$ which will play a role in the proof. Henceforth we will use the notation $y_i$ for both the tableau $y_i$ it defines and also the standard monomial associated it to.

\[
y_1 = \begin{tabular}{|c|c|c|c|c|c|c|}
\hline
1 &1	& 1&	2&	2&	2&	3\\ 
\hline
3&	3	&4	&4	&4	&5&	5\\   
\hline                
5&	6&	6&	6&	7&	7&	7 \\ 
\hline
\end{tabular} 
\ \ y_2 =  \begin{tabular}{|c|c|c|c|c|c|c|}
\hline
1&	1&	1&	2&	2&	2&	3\\
\hline
3&	3&	4&	4&	5&	5&	5\\ 
\hline                  
4&	6&	6&	6&	7&	7&	7\\
\hline
\end{tabular}
\ \ y_3= \begin{tabular}{|c|c|c|c|c|c|c|}
\hline 
1&	1&	1&	2&	2&	3&	3\\
\hline
2&	3&	4&	4&	4&	5&	5\\
\hline		
5&	6&	6&	6&	7&	7&	7\\
\hline
\end{tabular}
\]
\[
y_4=\begin{tabular}{|c|c|c|c|c|c|c|}
\hline
1&	1&	1&	2&	2&	3&	3\\
\hline
2&	3&	4&	4&	5&	5&	5\\               
\hline
4&	6&	6&	6&	7&	7&	7\\
\hline
\end{tabular}
\ \ y_5=\begin{tabular}{|c|c|c|c|c|c|c|}
\hline
1&	1&	1&	2&	3&	3&	3\\
\hline
2&	2&	4&	4&	4&	5&	5\\		
\hline
5&	6&	6&	6&	7&	7&	7\\
\hline
\end{tabular}
\ \ y_6=\begin{tabular}{|c|c|c|c|c|c|c|}
\hline						
1&	1&	1&	2&	3&	3&	3\\
\hline
2&	2&	4&	4&	5&	5&	5\\
\hline               
4&	6&	6&	6&	7&	7&	7\\
\hline
\end{tabular}
\]						
\[
y_7=\begin{tabular}{|c|c|c|c|c|c|c|}
\hline
1&	1&	1&	2&	2&	3&	3\\
\hline
2&	4&	4&	4&	5&	5&	5\\
\hline               
3&	6&	6&	6&	7&	7&	7\\
\hline
\end{tabular}
\ \ z_{20}=\begin{tabular}{|c|c|c|c|c|c|c|c|c|c|c|c|c|c|}
\hline
1&       1&       1&       1&       1&       1&       2&       2&       2&       3&       3&       3&       3&       3\\
\hline
2&       2&       2&       4&       4&       4&       4&       4&       4&       5&       5&       5&       5&       5\\
\hline
3&       5&       6&       6&       6&       6&       6&       6&       7&       7&       7&       7&       7&       7\\
\hline
\end{tabular}
\]

We first make
some simple observations.

\begin{observation}
\label{obs:firstcol}
Every semistandard tableau basis of $H^0(X(w_{3,7}), {\cal L}(7m \omega_3))^T$ begins with one of the following columns - ${[1, 2, 3], [1, 2,  4] [ 1, 2, 5], [1, 3, 4], [1, 3, 5]}$, and ends with the column $[3, 5, 7]$.
\end{observation}

\begin{proof}
We already noted above that the last column of every semistandard tableau basis element of $H^0(X(w_{3,7}), {\cal L}(7m \omega_3))^T$ is $[3, 5, 7]$.

Clearly, semistandardness forces that in the first row the leftmost $3m$ entries are filled with 1, and that in the last row the rightmost 6m entries are filled with 3m $6's$ followed by $3m$ 7's. So clearly the last entry of the first column cannot be 
6 or 7 otherwise we will have more 6's or 7's than permitted. The
second entry of the first column cannot be 5, otherwise the entire second row will have
only 5's a contradiction to the number of $5$'s present.
The second entry of the first column cannot be 4 for a similar reason - in that case the second row 
will only have 4's and 5's forcing at least one of them to occur more
than 3m times, a contradiction. This completes the proof.
\end{proof}
\qed

\begin{observation}\label{obs:m}
No semistandard tableau basis element of $H^0(X(w_{3,7}), {\cal L}(7m \omega_3))^T$ has the
following columns: $[1, 2, 7]$, $[1, 3, 7]$, $[1, 4, 7]$,$[1, 5, 6]$,$[1, 5, 7]$,
$[2,3,4]$,$[2, 3, 5]$, $[2, 3, 6]$,$[2, 3, 7]$,\\$[2, 4, 5]$,$[2, 5, 6]$,$[3, 4, 6]$, $[3, 5,
6]$.  The only columns containing a 6 are columns $[1, 2, 6]$, $[1, 3, 6]$,
$[1, 4,6]$ and $[2, 4, 6]$. There are exactly $m$ columns with
$[2, 4, 6]$ and at least $m$ columns with $[1, 4, 6]$. The only columns containing a 7 are $[2, 4, 7]$, $[2, 5, 7]$, $[3, 4, 7]$ and $[3, 5, 7]$ and there are at least $2m$ occurrences of columns $[2, 5, 7]$ and $[3, 5, 7]$. 
\end{observation}

\begin{proof}

If there is a column with $[1, x, 7]$, $x$ among $2,3,4,5,6$,
standardness forces that the 
entries in the first row to the left of this column are all 1's and the entries in the third row to its right are all 7. Then no matter where this column appears either the number of 1's or the number of 7's is incorrect. 

If there is a column with $[1,5,6]$,  then standardness forces the subsequent columns to all have a 5 in the second row and the columns preceding it to have a 1 in the topmost row. Then no matter where this column appears either the number of 1's or the number of 5's is not $3m$.

If $[2,3,4]$ occurs it is necessarily in column $3m+1$ appearing
immediately after the occurrence of all the columns containing 1
because it is lexicographically least among columns beginning with 2. But the entry
in the bottom row position $3m+1$ is 6 is a contradiction. The same
argument shows that $[2,3,5], [2,4,5]$ cannot occur. 
If $[2, 3, 6]$ or $[2, 3, 7]$  is present, the first row to the right of
this column and the second row to the left of this column contain only
2,3's yielding a total of $7m$ entries with 2 and 3, a contradiction.

If $[2, 5, 6]$ is present all columns to the right of this column will
have a 5 in the second row by standardness. But then all the $3m$
columns containing 7 will be of the type $[x, 5, 7]$, for some $x$. But
then the number of $5$'s is at least $3m+1$, a contradiction.

If the column $[3, 4, 6]$ is present then  the top row  to its right is
filled with 3's. So every column containing 7 in the bottom has 3 
as its topmost element. So the number of 3's is at least $3m+1$, a contradiction.

Now suppose the column $[3, 5, 6]$ is present. If it is in the left half
of the tableau, standardness will forces the number of 5's to be more
than what is allowed.  If it is in the right half of the tableau then all
entries in the second row to the right of this column are filled with
5's. So all the columns containing 7 in fact contain both 5 and
7. Again, the number of 5's is more than $3m+1$,  a contradiction.
  
The above argument shows that the column appearing immediately after
all the columns containing a $1$, is column $[2, 4, 6]$. It appears
before the $3m$ columns containing a 7. Since the tableau has no lexicographically 
larger column containing a 6, this column repeats till the appearance
of a 7. So it occurs exactly $m$ times. 

Now the remaining $2m$ columns containing a 6 in the last row occur to
the left of column number $3m+1$ which has a $[2, 4, 6]$.  Suppose there
are less than $m$ columns with $[1, 4, 6]$ in the tableau. Since $[1, 4,
6]$ appears to left of the column numbered $3m+1$, all the entries in
the second row to the left of first column labeled $[1, 4, 6]$ must have
2 or 3. So there are at least $2m+1$ 2's and 3's in the second
row. Now there are at least $4m$ locations in the first row to the right
of last 1 which have  only 2 or 3. So the total number of 2's and 3s
is at least $6m+1$, a contradiction. So we conclude that at least $m$ rows to the
left of column numbered $3m+1$ contain $[1, 4, 6]$.

We cannot have a $5$ in the first row. Since we can have a $5$ in the
third row only in positions $\{1,2,\ldots,m\}$, and the only  columns having a
$5$ in the second row are $[2, 5, 7]$ and $[3, 5, 7]$ it follows that we
need at least $3m-m$ columns with $[2, 5, 7]$ and $[3, 5, 7]$. 
\end{proof}
 \qed

\begin{lemma}
\label{lm:main}
Let $m \geq 2$. Every semistandard basis element of $R(m)$ is a product of a $y_i$ and an element
of $R(m-1)$, or is a product of $z_{20}$ and an element of 
$R(m-2)$.
\end{lemma}
\begin{proof}
Let $f$ be a semistandard basis element of $R(m)$. The proof follows a case by case analysis.
\begin{itemize}
\item[a] The first column of $f$ is $[1, 2, 3]$. 
By Observation ~\ref{obs:m} above we have at least $m$ columns with
$[1, 4, 6]$ and exactly $m$ columns labeled $[2, 4, 6]$.  Furthermore we can
have at most $(m-1)$ 5's in the last row of $f$. So we have at least
$2m+1$ columns in $f$ with $[2, 5, 7]$ and $[3, 5, 7]$, since these are the only
columns containing 5 in the second row.  The last column of $f$ is a $[3, 5,
7]$. If the remaining 2m, columns were all $[2, 5, 7]$ using Observation
~\ref{obs:m} the total number of 2's is at least $m+1+2m$ a contradiction.
(the $m$ 2's from columns with $[2, 4, 6]$, and one from the first
column having
$[1, 2, 3]$.  It follows that there are at least two columns with $[3, 5, 7]$.  
 
\begin{enumerate}[leftmargin=*,align=left]
\item Suppose $f$ has at least one column with $[2,5,7]$. Then we have
  one $[1,2,3]$, at least two $[1,4,6]$'s and one $[2,4,6]$, one $[2, 5,7 ]$ and 
two $[3,5,7]$'s. So the tableau $y_7$ appears as a
subtableau. The complement of this subtableau in $f$ is an element of $R(m-1)$. So $f$ is a 
product, of $y_7$ and an element from $R(m-1)$ as required.
 
\item $f$ has no $[2,5,7]$.  So we have at least $2m+1$
  columns in $f$ having $[3,5,7]$'s. Now the  remaining $7$'s can be 
made up from $[3,5,7]$'s or $[2,4,7]$'s or $[3,4,7]$'s. These cannot all come from $[3,5,7]$ and $[3,4,7]$ since the number of $3$'s in that case would be more than $(3m+1)$. So there is at least one $[2,4,7]$. Note that there are also at most $(m-1)$ $[2,4,7]$'s, $[3,4,7]$'s, and additional $[3,5,7]$'s in from column numbers $(4m+1)$
to $(5m-1)$. Now the number of 2's in row 1 is at most $(m+m-1)$ (from the
columns with $[2,4,6]$, and at most $m-1$ columns with $[2,4,7]$). So
we need at least $(m+1)$ 2's in the second row. In this case then the
second row of columns 1 to column $(m+1)$ contains only $2$.  In particular $[1,
2, 6]$ is present in $f$. Since there are at most $3m-1$
5's in the second row of $f$ (since we know there is a $[2,4,7]$), there is
at least one $5$ in the bottom row of $f$ in position $\{1,2,\ldots,m\}$,  forcing a
$[1,2,5]$ in $f$. 

Now the total number of $4$'s and $5$'s is $6m$. 
The total
number of $4$'s and $5$'s in the last row is at most $(m-1)$.    
We have
$m$ 4's from the $[2,4,6]$. The total number of $4$'s and $5$'s from
the columns containing $[2,4,7]$, $[3,4,7]$ and $[3,5,7]$ is at most
$3m$ . All of these can account for a total of $(5m-1)$ 4's and $5$'s coming from
these columns. Since we can have no more $5$'s in the second row, the
deficit $(m+1)$ needed must be made from 4's in the second row, in
fact occurring in columns numbered $m+2$ to $3m$. So we have at least
$3$ columns in $f$ with $[1,4,6]$'s. 

Taking stock, in $f$ we have 
one $[1,2,3]$, a $[1,2,5]$, a $[1,2,6]$, 3 columns with $[1,4,6]$, 2
columns with $[2,4,6]$, one $[2,4,7]$ and at least 5 $[3,5,7]$'s. So we see that the tableau indexing the basis vector $Z_{20}$ is a subtableau of $f$ and the complement of this subtableau in $f$ is an element of $R(m-2)$. 
$f$ is a product of $z_{20}$ and an element from $R(m-2)$. 
\end{enumerate}
  
\item[b] The first column of $f$ is $[1, 2, 4]$.  In this case there are at
  most $(m-1)$ $5$'s in the last row of $f$ and so there should be at least
  $(2m+1)$ columns in $f$ with $[2,5,7]$ and $[3,5,7]$.  

Notice that the $3m$ 6's cannot all come from columns with $[1,4,6]$ and $[2,4,6]$ alone. If that were the case we will have $3m$ 4's from these columns, and an additional $4$ from the first column, a contradiction. So at least one of the columns in $f$ with a $6$ has to be $[1,2,6]$ or $[1,3,6]$.

\begin{enumerate}[leftmargin=*,align=left]
\item Suppose a $[1,2,6]$ is present in $f$. Then it has to be in column $(m+1)$of $f$. Then we have $2$'s in the second row of $f$ in columns $1$ to columns $(m+1)$ by semi-standardness.  From Observation ~\ref{obs:m} we have $m$ $2$'s from the columns with $[2,4,6]$, so we can have at most $(m-1)$ columns with $[2,5,7]$. This means there are at least $2m+1 -(m-1)= m+2$ columns
with $[3,5,7]$, so we have at least $4$ columns with $[3,5,7]$. But
this means we have 
a $[1,2,4]$, a $[1,2,6]$, a $[1,4,6]$, a $[2,4,6]$, and three $[3, 5,
7]$, i.e. the tableau indexing the basis element $y_6$ as a subtableau in this case. 

\item Now suppose we do not have a $[1,2,6]$ in $f$ but have a $[1,3,6]$.  We claim a $[2,5,7]$ must appear. Notice that we can have at most $(m-1)$ $[2,4,7]$'s since we already have $(2m+1)$ 7's. Now there are at most $(m)$ $2$'s in the second row of $f$. But this means we have at most $m + m-1+m < (3m)$ 2's in $f$, a contradiction.  So we may assume we have at least one $[2,5,7]$ in $f$. We claim that we have at least 2 $[3,5,7]$'s in $f$, for otherwise we have $(2m)$
$[2,5,7]$'s. But we have more 2's than allowed since we have $(m)$ 2's from
the $[2,4,6]$ and a $2$ also from the $[1,2,4]$. So we conclude
we have a $[1,2,4], [1,3,6], [1,4,6], [2,4,6], [2,5,7],[3,5,7], [3, 5,
7]$.  So $y_4$ is a subtableau and we are done in this case. 
\end{enumerate}

\item[c] If the first column in $f$ is $[1,2,5]$.  If there are no $[1,2,6]$ or $[1,3,6]$ in $f$  then column $(m+1)$ must be $[1,4,6]$ and the first $m$ elements in the third row must be all $5$'s. But then the second and third row together have more than $(6m)$ 4's and 5's, a contradiction. So
either $[1,2,6]$ or  $[1,3,6]$ or, maybe,
both are present in $f$.

\begin{enumerate}[leftmargin=*,align=left]
\item Suppose first that $f$ has a $[1,2,6]$. Then the second row of $f$ has
at least $(m+1)$ 2's, and since we already have $m$ 2's from the $[2,4,
6]$' we can have at most $(m-1)$ $[2,5,7]$'s. This forces at least $(m+1)$
columns of $f$ to be $[3,5,7]$'s. 

If $f$ has a column with $[3, 4, 7]$ we see the tableau indexing $y_5$
is present as a subtableau of $f$.  If $f$ has no $[3,4,7]$ - we count the
number of 4's - we have $m$ from columns $[2,4,6]$.  We can
have at most $(2m-1)$ columns with$[1,4,6]$ since 2's occupies positions
1 up to $(m+1)$ in the second row. To make up the requisite 4's we need to
have at least one $[2,4,7]$. But then the number of $3$'s in the first
row of the tableau from $[3,5,7]$ is at most $(3m-1)$, so to make up
the requisite 3's,
there must be a $[1,3,6]$. In which case 
we see that the tableau indexing $y_3$ is present as a subtableau of $f$.

\item Suppose we only have a $[1,3,6]$ in $f$ and no $[1,2,6]$.
 Now the total number of columns in $f$ with $[1,2,5]$ and $[1,3,5]$ is 
$m$. We have exactly $2m$ $[2,5,7]$ and $[3,5,7]$ put together. If $f$ has no $[2,4,7]$, the remaining $m$ $7$'s
come from $[3,4,7]$. But then the tableau cannot have $[2,5,7]$ since semi-standardness will be violated.  So we only have $2m$ columns of $f$ with $[3,5,7]$. However this means we have $(3m+1)$ $3$'s a contradiction. So we may assume that we have a $[2,4,7]$.  We show then that there are at least two $[3,5,7]$, so the tableau indexing $y_3$ is a subtableau.  Otherwise there are $(2m-1)$ $[2,5,7]$'s.But we already have one $[2,4,7]$, one $[1,2,5]$ and $(m)$ [2,4,6]'s a contradiction to the numer of allowed $2$'s. 
\end{enumerate}

\item[d] If the first column of $f$ is $[1,3,5]$. Then all the 2's occur in the first row of $f$ and in columns $(3m+1)$ to $6m$ and there are $(m)$ 3's in the last $m$ columns of the first row. Since there are no 3's 
in the last row, the remaining $(2m)$ 3's must occur in the second row. Since 6 occurs in the last row in positions $(m+1)$ to $(4m)$, and a 1 occurs in the first row in columns 1 to $3m$, it follows that  $[1,3,6]$  is a column in $f$.
Now all the 4's occur in the second row, starting at position $(2m+1)$
and ending at position $5m$, after which we only have 5's in the
second row. Since the 2's in the first row of $f$ occur in positions $(3m+1)$
to $6m$ and the 7's occur in the bottom row in position $(4m+1)$ to $(7m)$
it follows that there is a column containing $[2,4,7]$ and a column
containing $[2,5,7]$. So $y_1$ is a subtableau of $f$.

\item[e] In case the first column is $[1,3,4]$, all the 2's occur in
  the first row, and so we have $(m)$ 3's in the first row appearing in
  the columns $[3,4,7]$ and $[3,5,7]$. So $f$ has at least $(m)$ columns with $[2,5,7]$.  Now there are $(2m)$ 3's in the second row and these must occur in positions $1$ to $2m$.
Since the last row has 6 in columns $(m+1)$ to $(4m)$ and the first row has a 1 in columns $1$ to $(3m)$, it follows that there is a column with filling $[1,3,6]$ in the given tableau.  It follows
that the tableau indexing $y_2$ is a subtableau of $f$ - we are done by induction as in the above case.
\end{itemize}

\end{proof} 

\begin{remark}
\label{rem:y5y7}
Let $\tau_1=[2,5,7]$, $\tau_2=[3,4,7]$, $\tau_3 =[2,4,7]$, $\tau_4=[3,5,7 ]$, $\tau_5=[2,3,7]$,
$\tau_6=[4,5,7]$. Consider the product of the Plucker coordinates $p_{\tau_1} p_{\tau_2}$. The straightening law gives us 
$p_{\tau_1}p_{\tau_2}=p_{\tau_3}p_{\tau_4}-p_{\tau_5}p_{\tau_6}$.
 On the Schubert variety $X(w_{3,7})$, $p_{\tau_6}=0$. 
So on $X(w_{3,7})$, $p_{\tau_1}p_{\tau_2}=p_{\tau_3}p_{\tau_4}$.  As a result $y_5 y_7 = z_{20}$.
\end{remark} 

\begin{theorem}
The ring $R=\oplus_{m\geq 0} H^0(X(w_{3,7}), {\cal L}(7m \omega_3))^T$ is generated in degree 1.
\end{theorem}

\begin{proof}\eat{ ${\mathbb C}[Y_1, Y_2,\ldots, Y_7]$ to $R$}

We continue to use the notation $y_1,\ldots, y_1, z_{20}$ for the semistandard tableau basis elements and the monomials they index.

The proof is by induction on $m$, the base case $m=1$ being
obvious. Now assume $m \geq 2$. Given any semistandard basis element of $H^0(X(w_{3,7}), {\cal L}(7m \omega_3))^T$, Lemma \ref{lm:main} shows that it can be written as a product of one of the $y_i$'s  and a semistandard basis element of 
$R(m-1)$, or as a product of $z_{20}$ in $R(2)$, and a semistandard basis element of $R(m-2)$. Because of Remark \ref{rem:y5y7} we have $z_{20} =y_5y_7$. So we can replace $z_{20}$ by $y_5y_7$. It follows by induction that every basis element of $R(m)$ is in the algebra generated by the $y_i$'s.
\end{proof}

It follows that there is a surjective ring homomorphism $\Phi: {\mathbb C}[Y_1, Y_2,\ldots, Y_7] \rightarrow R$, sending $Y_i$ to $y_i$.

Now let ${\cal I}$ be the two sided ideal generated by the following
relations in ${\mathbb C}[Y_1, Y_2,\ldots, Y_7]$.
\begin{subequations}
\label{eq:rels}
\begin{align}
Y_1Y_4&=Y_2Y_3- Y_2Y_7+Y_1Y_7 \label{first}\\
Y_1Y_5&=Y_3^2 - Y_3Y_7\label{second}\\	
Y_1Y_6&=Y_3Y_4 - Y_4Y_7\label{third}\\	
Y_2Y_5&=Y_3Y_4 - Y_3Y_7\label{fourth}\\
Y_2Y_6&=Y_4^2-Y_4Y_7\label{fifth}\\
Y_3Y_6&=Y_4Y_5\label{sixth}
\end{align}
\end{subequations}

\begin{theorem}
\label{thm:main}
The map $\Phi$ induces an isomorphism $\tilde{\Phi}: {\mathbb C}[Y_1, Y_2,\ldots, Y_7]/{\cal I} \simeq R$.
\end{theorem}
\begin{proof}
By explicit calculations one can check that the above relations hold with $Y_i$ replaced by $y_i$ - are in the kernel of $\tilde{\Phi}$. We omit these calculations. To complete
the proof we show we can use the above relations as a reduction
system. The process consists of replacing monomial $M$ in the $Y_i$'s 
which is divisible by a term $L_i$ on the left
hand side of one of the reduction rules $L_i = R_i$, by $(M/L_i)R_i$. Here $R_i$ is the right hand side of $L_i = R_i$. 

We show that the diamond lemma of ring theory holds for this reduction
system \cite{bergman2008diamond}. What this implies is that
any monomial in the $Y_i$'s reduces, after applying these reductions (in
any order, when multiple reduction rules apply) to a unique
expression in the $Y_i$'s, in which no term is divisible by a term
appearing on the left hand side of the above reduction system. 

We prove that the diamond lemma holds for this reduction system by looking at the reduction of the minimal overlapping 
ambiguities $Y_1Y_2Y_5, Y_1Y_2Y_6, Y_1Y_3Y_6$ and $Y_2Y_3Y_6$. We show in each case that the
final expression is unambiguous. It follows that any relation among $Y_i$'s
is in the two sided ideal ${\cal I}$ generated by the above
relations. This proves that the map $\tilde{\Phi}$ constructed above is injective,.

To complete the proof we look at the reductions of overlapping
ambiguities.

$Y_1Y_2Y_5$ - using rule \ref{second} above we get $Y_2(Y_3^2-Y_3Y_7)= Y_2Y_3^2
-Y_2Y_3Y_7$ which cannot be reduced further. 
 On the other hand using rule \ref{fourth} above we get $Y_1(Y_3 Y_4
-Y_3Y_7)$. Now this can be further reduced using rule \ref{first} and we get
$Y_3 Y_1 Y_7 + Y_3 Y_2Y_3 -Y_3 Y_2 Y_7 -Y_1Y_3Y_7$ and this is equal
to $Y_2Y_3^2 -Y_2Y_3Y_7$. The reduction is unique in this case.

Likewise one can show that $Y_1Y_2Y_6$ reduces to the unique
expression $Y_2Y_3Y_4-Y_2Y_4Y_7$. And $Y_2Y_5Y_6$ reduces to $Y_4^2Y_5
- Y_4Y_5Y_7$ and $Y_2Y_3Y_6$ reduces to $Y_3Y_4^2 - Y_3^2Y_4Y_7$
completing the proof.
\end{proof}

\section{Deodhar decomposition to compute quotients of Richardson varieties}\label{s.deodhar}

This section is again motivated by the question of understanding GIT quotients of Richardson varieties in the Grassmannian. In Section \ref{s:taurngen} we proved some results on quotients of Richardson varieties. A natural strategy to understand the GIT quotient is
to take a stratification of a Richardson variety, understand what the GIT quotient of each strata is, and also understand how the GIT quotients of these strata patch up. Such a stratification of the open cell of a Richardson variety was given by Deodhar
 \cite{deodhar1985some}. This was to be our starting point. Working with small examples we believed that the restriction of a $T$-invariant section to the open cell would be a homogenous polynomial and that this would lead us to discover the equations defining the GIT quotient of a Richardson variety. However we soon realized that sections may not restrict to homogenous polynomials on the open cell, that the issue is more subtle.
We have necessary conditions which guarantee when sections restrict to homogenous polynomials on the open cell. This is Lemma \ref{lm:homog}. To state the Lemma and also the proof we need to introduce the Deodhar decomposition and some more notation and theorems about Deodhar decomposition of Richardson varieties on the Grassmannian. We do that in the next Subsection \ref{sub:deod}.  We use the Deodhar decomposition to study the GIT quotients of Richardson varieties in $X(w_{3,7})$ in Section \ref{s:deodhar}. Although all these results follow from the results in Section \ref{s:taurngen} we prove them again since this can be done by explicit calculations. Finally we show that the GIT quotient of $X(w_{3,7})$ is a rational normal scroll. We were unable to complete this proof
using only information about the GIT quotients of Richardson varieties in $X(w_{3,7})$. Instead we show that the equations defining the GIT quotient is a determinantal variety.

\subsection{Deodhar decomposition}
\label{sub:deod}
In \cite{deodhar1985some} Deodhar considered the intersection in $G/B$ of the open cell in a Schubert variety with the open cell of an opposite Schubert variety.  For $v, w \in W$, define the Richardson strata\footnote{this terminology is not standard. What we have called strata is sometimes called a Richardson variety}
$$ R^v_w = (BwB/B) \cap (B^{-}vB/B)$$

Note that this is not the same as the definition of a Richardson variety (see for example \cite{brion2003geometric}). Recall that for $v,w \in 
W$ a Richardson variety in $G/B$ was defined to be the intersection of $X(w) \cap X^v$. Since both $X(w)$ and $X^v$
contain the intersection of $(BwB/B) \cap (B^{-}vB/B)$ it is clear that
$R^v_w \subseteq X^v_w$.  And so Richardson strata is empty if $v
\not \leq w$ and the closure of $R^v_w$ is $X^v_w$.

In \cite{deodhar1985some} Deodhar gave a refined decomposition of a Richardson strata in $G/B$ into disjoint locally closed subvarieties of a Schubert variety.  We follow the notation from Marsh and Reitsch\cite{marsh2004parametrizations}, and Kodama and Williams \cite{kodama2013deodhar}. The definitions and examples are taken verbatim from \cite{kodama2013deodhar} since it is their notation and set up that we use in our proofs. 

Fix a reduced decomposition ${\bf w}=s_{i_1}s_{i_2}\cdots s_{i_m}$. 
We define a subexpression ${\bf v}$ of ${\bf w}$ to be a word obtained from the reduced expression ${\bf w}$ by replacing some of the factors with 1. For example, consider a reduced expression in $S_4$, say $s_3s_2s_1s_3s_2s_3$. Then $s_3s_2 1 s_3s_2 1$ is a subexpression of $s_3s_2s_1s_3s_2s_3$ . Given a subexpression ${\bf v}$, we set $\mathbf{v}_{(k)}$ to be the product of the leftmost $k$ factors of $\mathbf{v}$, if $k \geq 1$, and set $\mathbf{v}_{(0)} = 1$. The following definition was given in \cite{marsh2004parametrizations} and was inspired from Deodhar's paper \cite{deodhar1985some}.

\begin{definition} Given a subexpression ${\bf v}$ of a reduced expression ${\bf w} = s_{i_1} s_{i_2} \cdots s_{i_m}$ , we define 
\begin{eqnarray*}
J_{\bf v}^{\circ}  & := &\{k \in \{1,...,m\} | \mathbf{v}_{(k-1)} < \mathbf{v}_{(k)}\}\\
J_{\bf v}^{\square} &:= & \{k \in \{1,...,m\} | \mathbf{v}_{(k-1}) = \mathbf{v}_{(k)}\}\\
J_{\bf v}^{\bullet}& := & \{k \in \{1,...,m\} | \mathbf{v}_{(k-1)} > \mathbf{v}_{(k)}\}\\
\end{eqnarray*}
\end{definition}

The expression ${\bf v}$ is called non-decreasing if $\mathbf{v}_{(j -1)} \leq \mathbf{v}_{(j)}$ for all $j = 1, \ldots , m$, and in this case $J_{\bf v}^{\bullet} = \emptyset$.

The following definition is from \cite[Definition 2.3]{deodhar1985some}.

\begin{definition} (Distinguished subexpressions). A subexpression ${\bf v}$ of ${\bf w }$ is called distinguished if we have  
\[ \mathbf{v}_{(j)} \leq \mathbf{v}_{(j-1)} s_{i_j} \ \forall \  j  \in \{1,\ldots,m\}\]
\end{definition}

In other words, if right multiplication by $s_{i_j}$ decreases the length of $\mathbf{v}_{(j-1)}$, then in a distinguished subexpression we must have $\mathbf{v}_{(j)} = \mathbf{v}_{(j-1)}s_{i_j}$ .

We write ${\bf v} \prec {\bf w}$ if ${\bf v}$ is a distinguished subexpression of ${\bf w}$.

\begin{definition} (Positive distinguished subexpressions). 
\label{def:posexp}
We call a subexpression ${\bf v}$ of ${\bf w}$ a positive distinguished subexpression (or a PDS for short) if
$\mathbf{v}_{(j-1)} < \mathbf{v}_{(j-1)}s_{i_j}$, for all $j\in \{1, . . . , m\}$.

\end{definition}

Reitsch and Marsh \cite{marsh2004parametrizations} proved
\begin{lemma}
\label{lem:uniquepos}
 Given  $v \leq w$ and a reduced expression ${\bf w}=s_{i_1}\cdots s_{i_m}$for $w$, there is a unique PDS ${\bf v}^+$ for $v$ in ${\bf w}$.
\end{lemma}

We now describe the Deodhar decomposition of the Richardson strata.
Marsh and Rietsch \cite{marsh2004parametrizations} gave explicit parameterizations for each Deodhar component, identifying each one with a subset in the group. Much of this appears implicitly in Deodhar's paper, but we refer to \cite{marsh2004parametrizations} for our exposition because these statements are made explicit there and the authors make references to Deodhar's paper wherever needed.

\begin{definition} \cite[Definition 5.1]{marsh2004parametrizations} 
\label{Gwv}
Let ${\bf w} = s_{i_1} \cdots s_{i_m}$ be a reduced expression for $w$, and let ${\bf v}$ be a distinguished subexpression. Define a subset ${\bf G_{w}^{v}}$ in $G$ by
\[ \mathbf{G_{w}^v} :=
\big{\{} g = g_{1}g_{2}\cdots g_{m}
\big{|} \begin{cases}
g_{l}=x_{i_l}(m_l) s_{i_l} &  \text{if $l \in J_{\bf v}^{\bullet}$},\\
g_{l}=y_{i_l}(p_l) &  \text{if $l \in J_{\bf v}^{\square}$};\\
g_l=s_{i_l} & \text{if $l \in J_{\bf v}^{\circ}$}
\end{cases}
\big{\}}
\]
\end{definition}

From\cite[Theorem 4,2]{marsh2004parametrizations} there is an isomorphism from ${{\mathbb C}^{*}}^{|J^{\square}_{\bf v}|} \times {\mathbb C}^{|J^{\bullet}_{\bf v}|} $ to ${\bf G^{v}_{w}}$. 

\begin{definition}(Deodhar Component) 
\label{def:compo}The Deodhar component $\mathbf{{\cal R}^v_w}$ is the image of ${\bf G^v_w}$ under the map ${\bf G^v_w} \subseteq U^{-} v B \cap B w B \rightarrow G/B$, sending $g$ to $gB$.
\end{definition}

Then from \cite[Theorem 1.1]{deodhar1985some} one has \cite[Corollary 1.2]{deodhar1985some}, also from Deodhar.
\begin{theorem} 
\label{deodhardecomposition}$R^v_w = \bigsqcup_{ \mathbf{v} \prec \mathbf{w}} \mathbf{{\cal R}^v_w}$ the union taken 
over all distinguished subexpressions ${\bf v}$ such that $\mathbf{v}_{(m)}=v$. The component 
 $\mathbf{{\cal R}^{v^+}_w}$ is open in $R^v_w$. 
\end{theorem}

Naturally when one is talking of the Deodhar decomposition of a Richardson strata in $G/P_{\hat{\alpha_r}}$, one can take the projections of the components in $G/B$ into $G/P_{\hat{\alpha_r}}$.  In \cite[Proposition 4.16]{kodama2013deodhar} the authors show that the Deodhar components of a Richardson strata in
$G/P_{\hat{\alpha_r}}$ are independent of $\mathbf{w}$ and only depends upon $w$. This follows from the observation that any two reduced decompositions ${\mathbf w}$ and ${\mathbf w'}$ of $w$ are related by a sequence of commuting transpositions $s_is_j = s_j s_i$.

\subsection{Quotients of Deodhar components in $X(w_{3,7})$}\label{s:deodhar}
Let us fix a reduced decomposition ${\bf w_{3,7}}=s_2
s_1s_4s_3s_6s_5s_2s_4s_3$ for the Weyl
group element $w_{3,7}$ with $X(w_{3,7})$ being the minimal Schubert
variety in $\grass{3}{7}$ admitting semistable points. In this section
we describe the GIT quotients of Richardson varieties in $X(w_{3,7})$ by
computing the various Deodhar strata in this Schubert variety 
and analyzing their quotients.  It will be useful to recall Definition \ref{Gwv} and the notation developed in Subsection \ref{sub:deod}. 

We begin with a corollary to Theorem \ref{thm:projnormal}
\begin{corollary}
\label{rem:prj}
The GIT quotient of Richardson varieties in $X({w_{3,7}})$ is projectively normal with respect to the descent of the $T$ linearized line bundle ${\cal L}(7 \omega_3)$.
\end{corollary}
\begin{proof}
Let $X^v_{w_{3,7}}$ be a Richardson variety in $X(w_{3,7})$.  From the proof of \cite[Proposition 1 ]{brion2003geometric} it follows that $H^0(X(w_{3,7}), {\cal L}(\omega_3)^{\otimes m})$ $\rightarrow$ $H^0(X^v_{w_{3,7}}, {\cal L}(\omega_3)^{\otimes m})$ is surjective. Since $T$ is linearly reductive it follows that
the map  $H^0(X(w_{3,7}), {\cal L}(\omega_3)^{\otimes m})^T$ $\rightarrow$ $H^0(X^v_{w_{3,7}}, {\cal L}(\omega_3)^{\otimes m})^T$
is also surjective.  From Theorem \ref{thm:projnormal} we know that the polarized variety $(\xtaumodt{3}{7}^{ss}_T({\cal L}(7 \omega_3)), \tilde{{\cal L}}(7 \omega_3))$ is projectively normal. Since $\richmodt{w_{3,7}}{v}{\mathcal{L}(7\omega_3)}$ is normal it follows that the GIT quotient of $X^v_{w_{3,7}}$ is projectively normal with respect to the descent line bundle.
\end{proof}

\begin{lemma}
Let $v=s_2 s_4 s_3$. Then $\richmodt{w_{3,7}}{v}{{\cal L}(7 \omega_3)}$ is a point.
\end{lemma}

\begin{proof}
\eat{This is clear from \ref{lm:point} since $w_{3,7}v^{-1}$ is the Coxeter
 word $s_2s_1s_4s_3s_6s_5$}
 The only torus-invariant section of $H^0(X(w_{3,7}),{\cal L}(7 \omega_3))$
 which is non zero on $X^v_{w_{3,7}}$ is the section $y_1$. Consider the
 Deodhar component of $X^v_{w_{3,7}}$ corresponding to the subexpression ${\mathbf
   v}=111111s_2 s_4 s_3$. This
section evaluated on a matrix in $\mathbf{G^{v}_{w_{3,7}}}$ is
  $p_1p_2^4p_3^2p_4^5p_5^3p_6^6$. Using the reduced expression for ${\bf w_{3,7}}$,  note that the weight of this monomial is $\alpha_2 + 4\alpha_1+2 \alpha_4+5 \alpha_3 + 3 \alpha_6 + 6 \alpha_5$.
\end{proof}

\begin{lemma}
let $v=s_2s_3$. Then $\richmodt{w_{3,7}}{v}{{\cal L}(7 \omega_3)}$ is isomorphic to ${\mathbb P}^1$ and the descent of ${\cal L}(7 \omega_3)$ is ${\cal O}(1)$.
\end{lemma}
\begin{proof} On the open Deodhar component corresponding to the 
reduced subexpression $1 1 1 1 1 s_2 1 s_3$ the only nonzero $T$-invariant standard monomials of shape $7\omega_3$ are $y_1$,
$y_2$.  These are algebraically independent.  The lemma follows now from Corollary \ref{rem:prj}.
\end{proof}

\begin{lemma}
Let $v = s_4 s_3$. Then $\richmodt{w_{3,7}}{v}{{\cal L}(7 \omega_3)}$ is isomorphic to ${\mathbb P}^1$ and the descent of ${\cal L}(7 \omega_3)$ to the GIT quotient is ${\cal O}(2)$.
\end{lemma}
\begin{proof}
The three nonzero sections on the open Deodhar cell
corresponding to the subexpression ${\mathbf v}=1111111s_4s_3$ are
$y_1,y_3, y_5$. Let $p=p_1p_2^4p_3^2p_4^5p_5^3p_6^6p_7^5$. Note
that $p$ is nowhere vanishing on the Deodhar cell. Let 
$X=(p_1+p_7), Y=p_1$. It can be checked that on the open Deodhar cell
$y_1$ evaluates to $pX^2$, $y_3$ to $pXY$ and $y_5$ to $pY^2$. The lemma follows from Corollary \ref{rem:prj}.
\end{proof}

\begin{lemma}
Let $v = s_3$. Then $\richmodt{w_{3,7}}{v}{{\cal L}(7 \omega_3)}$ is isomorphic to ${\mathbb P}^1
\times {\mathbb P}^1$ and the
descent of the line bundle to the GIT quotient is ${\cal O}(2) \boxtimes {\cal O}(1)$.
\end{lemma}
\begin{proof}
Let $p=p_1p_2^4p_3^2p_4^5p_5^3p_6^6p_7^5p_8^6$. 
Let $A=p_3, B=p_3+p_8$. Let $X=(p_1+p_7), Y=p_1$.
Note that $p_3$ and $p_8$ are algebraically independent and so $A, B$ are
algebraically independent.
Since $p_1$ and $p_7$ are algebraically independent so are $X,Y$.

First note that on the open Deodhar cell corresponding to the distinguished
subexpression $1 1 1 1 1 1 1 1 s_3$, $p$ is nowhere vanishing. On the open cell
$y_1$ evaluates to $pBX^2$.
The section $y_5$ evaluates to $pBY^2$. On the other hand $y_3$ evaluates to 
$pXYB$.
Likewise $y_2$ evaluates to $pAX^2$, $y_6$ evaluates to $pAY^2$.
And $y_4$ evaluates to $pXYA$. 

So, upto a multiple of $p$, the sections $y_2,y_4,y_6,
y_1, y_3,y_5$ can be respectively written as  $(X^2A,XYA,
Y^2A, X^2B, XYB, Y^2B)$. Using Corollary \ref{rem:prj} it follows that the GIT quotient is isomorphic to
${\mathbb P}^1 \times  {\mathbb P}^1$ embedded as ${\cal O}(2)
\boxtimes {\cal O}(1)$. 
\end{proof}

In the next lemma we give conditions guaranteeing when a section 
of the line bundle ${\cal L}(n \omega_r)$ on $X^v_w$ restricts to a homogenous polynomial on the Richardson strate in $X^v_w$.
  
\begin{lemma}
\label{lm:homog}
Let $u \in W, v \in W^{S\setminus\{\alpha_r\} } $ be such that $w =uv \in W^{S\setminus\{\alpha_r\} }$ and $l(uv) = l(u) + l(v)$. Fix a reduced expression for $u =s_{i_1} \cdots s_{i_k}$ and a
reduced expression for $v= s_{i_{k+1}} \cdots s_{i_m} $ such that $\mathbf{w} = s_{i_1} \cdots s_{i_k}.s_{i_{k+1}} \cdots s_{i_m}$ is a reduced expression for $w$. Consider $\mathbf{v} =1\cdots1s_{i_{k+1}}.\cdots s_{i_m}$, a distinguished subexpression of $\mathbf{w}$. ${\cal R}^{\mathbf v}_{\mathbf w}$ is the unique open Deodhar component of $R^v_w$. The restriction of any section $s\in H^0(X^v_w, {\cal L}(n \omega_r))$ to $R^v_w$ is a homogeneous polynomial in $p_1,p_2,\cdots,p_k$ having degree = ht $v(n\omega_r)$.
\end{lemma}

\begin{proof}

Note that ${\mathbf v}$ is the unique positive distinguished subexpression for $v$ in ${\mathbf w}$ and so $\mathcal{R}_\mathbf{w}^\mathbf{v}$ is the unique open Deodhar component of $R^v_w$. 

Matrices in 
${\bf G^v_w}$ are of the form $y_{i_1}(p_1) y_{i_2}(p_2) 
\ldots y_{i_k}(p_k) s_{i_{k+1}}\cdots s_{i_m}$. From this identification we see that the section $s$ restricted to this Deodhar component is 
$s_{|\mathcal{R}_w^v} =  \sum_{\mathbf{m}} a_{\mathbf{m}} p_1^{m_1} \ldots p_k^{m_k}$ where $\mathbf{m} = (m_1,..m_t)$. If $a_{\mathbf{m}} \neq 0$ then $wt (s) = wt (p_1^{m_1} \ldots p_k^{m_t}) = v(n\omega_r)$. In particular $deg (p_1^{m_1} \ldots x_k^{m_k}) =  ht (v(n\omega_r))$ . 
\end{proof}

Finally we prove 
\begin{theorem}
The polarized variety $(\richmodt{w_{3,7}}{id}{{\cal L}(7 \omega_3)}, \tilde{{\cal L}}(7 \omega_3))$ is a rational
normal scroll.
\end{theorem}
\begin{proof}
The relations \ref{eq:rels} given before Theorem \ref{thm:main} describe the homogenous ideal defining the polarized variety. These defining relations can be written succinctly in a matrix form 
\[ rank\begin{pmatrix}
Y_1 & Y_3 & Y_4 & Y_2 \\
Y_3-Y_7& Y_5& Y_6& Y_4 -Y_7\\
\end{pmatrix} \leq 1
\]
For example the minor corresponding to the first two columns above gives us $Y_1Y_5=Y_3^2-Y_3Y_7$, which is \ref{second},  and the minor corresponding to columns 1 and 3 gives relation \ref{third} shown there. 
So the polarized variety is a rational normal scroll.
\end{proof}

\section{Projective normality of the GIT quotient of $G_{2,n}$}\label{s:pnormalg2n}

In this section we study the GIT quotient of $G_{2,n}$ with respect to the $T$-linearized line bundle $\mathcal{L}(n\omega_2)$ for $n$ odd. As mentioned earlier, this line bundle descends to the quotient, and it is well known that the polarized variety $((\gmnmodt{2}{n})_T^{ss}({\cal L}(n\omega_2)), \tilde{{\cal L}}(n \omega_2))$ is projectively normal (see, \cite{howard2005projective}, \cite{Ke}). We give an alternate proof. It is not clear to us whether this result extends to GIT quotients of higher rank Grassmannian's. To the best of our knowledge this question is open.\eat{ We were hoping that the techniques in this proof would allow us to settle the question one way or the other, but we couldn't succeed. Nevertheless} We believe that it is this kind of combinatorics which will be required to settle the question.

We follow the strategy outlined in Remark \ref{rem:projnormal}. Defining $R(m)$ to be $H^0(G_{2,n}, {\cal L}(n \omega_2)^{\otimes m})^T$  we show that $R(1)^{\otimes m} \rightarrow R(m)$ is surjective.

Let $p_{\underline{\tau}} = p_{\tau_1}p_{\tau_2} \ldots p_{\tau_{mn}}$ be a standard monomial in $R(m)$ and let $T_{\underline{\tau}}$ be the tableau  associated to this monomial.

Denote the columns of $T_{\underline{\tau}}$ by $C_1, C_2, \cdots ,C_{mn}$ with $C_i = [a_i, b_i]$. Our idea is to extract from the tableau $T_{\underline{\tau}}$ a semistandard Young subtableau $T_{\underline{\mu}}$, with each integer $1,2,\ldots,n$ appearing exactly two times.  Then the monomial $p_{\underline{\mu}}$ corresponding to this subtableau would be a zero weight vector in $R(1)$, and the monomial corresponding to the remaining columns in $T_{\underline{\tau}}$ would be a monomial $p_{\underline{\nu}} \in R(m-1)$.  If we were to succeed in doing this, we could write $p_{\underline{\tau}}$ 
as a product of  $p_{\underline{\mu}}$ and $p_{\underline{\nu}}$, and we would be 
done by induction on $m$. \eat{However that is not true as the monomial corresponding to tableau $z_{20}$ from Section \ref{s:tau37} shows.  If we were able to extract a subtableau in $z_{20}$ with each integer appearing three times, such a tableau would necessarily have to be one of the tableaus $y_{1}, y_{2},\ldots, y_{7}$. One can easily check that none of these tableaus occur as subtableaus of $z_{20}$}Since we were unable to do this directly we use straightening laws on tableaus to show that  $p_{\underline{\tau}}$ can be written as a sum of products of elements in $R(1)$.

Let $p_{\underline{\mu}} :=  p_{\tau_1}p_{\tau_{m+1}} \ldots p_{\tau_{mn - m+1}}$ and $p_{\underline{\nu}} = \widehat{p_{\tau_1}} p_{\tau_2}.p_{\tau_3}\cdots p_{\tau_m} \widehat{p_{\tau_{m+1}}} \cdots p_{\tau_{mn}}$. Here $\widehat{p}$ indicates that the corresponding term is omitted. 
Clearly $p_{\underline{\tau}} = p_{\underline{\mu}} p_{\underline{\nu}} $.

Let $T_{\underline{\mu}}$ and $T_{\underline{\nu}}$ denote the corresponding tableaus.

\begin{definition} 
\label{def:defect}
An integer $i$ is defected if $i$ appears an odd number of times in $T_{\underline{\mu}}$. Denote the set of defected integers by $\mathcal{D}$.
\end{definition}

\begin{lemma}
 All integers  in $\{1,2,\ldots,n\}$ occur in $T_{\underline{\mu}}$.
\end{lemma}
\begin{proof}
Every integer $j$ has to appear at least $m$ times in one of the rows of $T_{\underline{\tau}}$.  Because $T_{\underline{\tau}}$ is semistandard $j$ appears consecutively, so there is a column $C_i$ with $i \equiv 1 \pmod m$ containing $j$.
\end{proof}

\begin{lemma} 
There are even number of defected integers.
\end{lemma}
\begin{proof}
$T_{\underline{\mu}}$ has $2n$ boxes and all the integers appear in $T_{\underline{\mu}}$. Each integer which is not defected appears twice. The number of times a defected integer appears is odd, so there are an even number of defected integers.
\end{proof}

Before we prove the next lemma we set up some notation and make some  observations. 

Let $f_i$ (respectively, $l_i$) be such that $C_{f_i}$(respectively, $C_{l_i}$) is the column in which $i$ appears for the first (respectively, last) time in the bottom row of $T_{\underline{\tau}}$ . Similarly define $f^i$ and $l^i$ with respect to occurrences of $i$ in the top row.

\begin{observation}
\label{obs:modm}
$f_i \equiv x+1  \pmod m$ if and only if  $f^i \equiv m-x+1 \pmod m$. In particular if $f_i\equiv\ 1 \pmod m$ if and only if $f^i \equiv 1  \pmod m$ and in this case $i$ appears at least two times in  $T_{\underline{\mu}}$. 
\end{observation}
\begin{proof}
Each integer less than $i$ appears $2m$ times and occurs in the top row in columns
before column $f^i$ and in the bottom row in columns before column $f_i$. The total number of positions for numbers from $1$ to $i-1$ is therefore a multiple of $m$. 
If $f_i$ is  $am+1+x$, then the number of boxes to the left of this column in the bottom row is $am+x$.  So $f^i$ must $bm+m-x+1$ for some $b$ so that the number of positions for integers $1$ to $i-1$ is $bm + m -x$ as needed.

The last statement follows since $T_{\underline{\mu}}$ is constructed by taking only columns numbered $1 \pmod m$ in $T_{\underline{\tau}}$
\end{proof}

Let $\mathcal{D} = \{i_1,i_2,\cdots,i_{2l}\} $ where $i_1 < i_2 \cdots < i_{2l}$.
 
 \begin{lemma} 
\label{lm:alternate}
Let $i_j \in \mathcal{D}$. In $T_{\underline{\mu}}$, $i_j$ appears 3 times if $j$ is odd and $i_j$ appears once if $j$ is even.
 \end{lemma}
 \begin{proof}
 
We show that two consecutive defected integers cannot both appear 3 times 
nor can they both appear once.  And then we show that the first integer which is defected appears 3 times.

Let us assume that some integer $i_j$ which is defected appears 3 times. W.l.o.g we may assume that it appears 2 times in the top  row and appears once in the bottom row. Assume that the next defected integer $i_{j+1}$ also appears 3 times. We prove it in the case when $i_{j+1}$ appears 2 times in the top row and once in the bottom row. The proof in the other case is similar.

Assume that the positions of $i_j$ (resp $i_{j+1}$) in $T_{\underline{\tau}}$ which contribute to its two occurrences in the top row of $T_{\underline{\mu}}$ are $(a-1)m+1, am+1$ (resp $bm+1, (b+1)m+1$). 
Likewise, assume that the positions of $i_j$, (respectively, $i_{j+1}$) in $T_{\underline{\tau}}$ contributing to the bottom row in $T_{\underline{\mu}}$ are $cm+1$ (respectively, $dm+1$). Clearly $c < a-1$ and $d < b$.
Let $x$ be the number of $i_j$ to the right of  position $am+1$ in the top row of $F$ and $z$ be the number of $i_j$ to the right of $cm+1$ in the bottom row of $T_{\underline{\tau}}$. Similarly let $y$ denote the 
number of $i_{j+1}$ to the left of position $bm+1$ in the top row and $w$ be the number of $i_{j+1}$ to the left of position $dm+1$ in the bottom row of $T_{\underline{\tau}}$. Clearly $x+z \leq m-2$ and $y+w \leq m-2$. 

Now $i_{j+1} = i_j +1$ is not possible. Because the number of $i_{j+1}$ in the top row is then at least $2m-x$ and the number of $i_{j+1}$ in the bottom row is at least $m-z$ a contradiction to the number of $i_{j+1}$ in $T_{\underline{\tau}}$, since $x + z \leq m-2$.

So let us assume that $i_{j+1} > i_j +1$. Now there are $i_{j+1} - i_j -1$ integers in between $i_j$ and $i_{j+1}$ which are not defected. Hence in $T_{\underline{\tau}}$ each of these integers occurs in exactly two positions which are in positions $1 \pmod m$. Hence the 
number of positions which are $1 \pmod m$ between the positions $am+1, bm+1$ and between
$cm+1, dm+1$ is exactly $2(i_{j+1} - i_j -1)$.  But this count is also equal to
$(b-a-1) + (d-c-1)$. Hence $b+d - a -c -2 = 2(i_{j+1} - i_j -1)$. Or $b+d -a-c=2((i_{j+1} - i_j)$. The total number of positions available for integers in the range $i_j+1$ to $i_{j+1}$ is exactly $bm -am -x -y -1 + dm-cm - z -w -1$ which is $m(b+d -a -c -2) - (x + y+z+w)$. Since each integer in this range appears exactly $2m$ times, and since $b+d -a-c=2((i_{j+1} - i_j)$ it follows that $x + y + z+w $ is $0$ modulo $2m$. If any of them is non zero this is impossible since $x+z \leq m-2$ and $y+w \leq m-2$. Suppose all of $x,y,z,w$ are zero. Then the positions $am+1+1$ to $bm$ and
 $cm+1+1$ to $dm$ are available for the integers $i_{j}+1,\ldots,i_{j+1}-1$.  This is 
$(b+d-a-c)m -2$ positions in all, which is also $2(i_{j+1}-i_j)m -2$ positions. 
But this is more positions than are required, since we have $i_{j+1} - i_j -1$ numbers each occurring $2m$ times - we require only $2m(i_{j+1} - i_j -1)$ positions. 

Next we show that if $i_j$ appears with defect 1 then $i_{j+1}$ appears with defect 3. W.l.o.g assume that $i_j$ appears in the top row in a column numbered 
$1 \pmod m$. So we know that $i_j$ appears less than $m$ times in the bottom row.

Assume that $f^{i_j}$ is $am + x+1$ for some $1 \leq x \leq m-1$. Then $f_{i_j}$ is
$bm + m-x+1$ for some $b$. Now since $i_j$ does notoccur in a column numbered $1  \pmod m$ in the bottom row, it follows that the number of $i_j$ in the bottom is at most $x$, so the number of $I_j$ in the top row is at least $2m-x$. But since there is only one occurrence of $i_j$ in a column numbered $1 \pmod m$,  there are at most $2m-x$ occurrences of $i_j$ in the top row.  It follows that the top row has exactly $2m-x$ occurrences of $i_j$ and the bottom row has exactly $x$ occurrences of $i_j$. So $l^{i_j} = l_{i_j}= 0  \pmod m$. Hence each integers between $i_j$ and $i_{j+1}$ which is not defected starts at a position which is $1  \pmod m\ $ on the top and ends at a $0  \pmod m \ $ position in the top and bottom rows (if it occurs in them). So the $f_{i_{j+1}}$ is forced to be 
$1 \pmod m$ and so $f^{i_{j+1}}$ is also $1 \pmod m$ by Observation \ref{obs:modm}. Since it is defected it occurs once more in a column numbered $1  \pmod m$.

We show that the first defected integer occurs 3 times to complete the proof. Suppose that $i_1$ occurs only once in $T_{\underline{\mu}}$. Then it occurs strictly more than $m$ times in the top or bottom row of $T_{\underline{\tau}}$.  W.l.o.g it occurs strictly more then $m$ times in the top row of 
$T_{\underline{\tau}}$, and say occurs in column $am+1$. Suppose $i_1$ makes its first appearance in $T_{\underline{\tau}}$ in column $(a-1)m+1+j$ for $0 < j \leq m-1$. Since it occurs only once in a column numbered $1 \pmod m$, the total number of occurence of $i_1$ in the top row of $T_{\underline{\tau}}$ is
at most $2m-j$. So it occurs in the bottom row of $T_{\underline{\tau}}$ as well. Now suppose its first occurence in the bottom row of $T_{\underline{\tau}}$ is in column $bm+1+k$, for $0 < k \leq m-1$. Since all integers less than $i_1$ occur $2m$ times in $T_{\underline{\tau}}$, it follows that $j+k = 0 \pmod m$.  But $j+k <2m$ and so $j+k=m$. Now each integer less than $i_i$ is not defected and so appears twice in
$T_{\underline{\tau}}$ in columns numbered $1 \pmod m$. The number of such columns available is $a-1+b$ and since this has to be even, $a+b$ must be odd. But then the total number of positions available for integers less than $i_1$ in $T_{\underline{\tau}}$  is $(a-1+b)m + j + k$ which is $(a+b)m$, an odd multiple of $m$. But each integer less than $i_1$ appears $2m$ times in $T_{\underline{\tau}}$, a contradiction to the number of available positions being an odd multiple of $m$.     
\end{proof}
  
 \begin{lemma} If $j$ is odd, $i_j$ appears in the top and bottom row of $T_{\underline{\mu}}$
 \end{lemma}
 \begin{proof}
 For $j$ odd we have $i_j$ appears thrice in $T_{\underline{\mu}}$. If all of them appear consecutively in  $T_{\underline{\mu}}$ then the number of $i_j$ is  $T_{\underline{\tau}}$ would be greater than $2m$, a contradiction. 
 \end{proof}
 
\begin{notation}
Let $T^k_{\underline{\tau}}$ be the subtableau of $T_{\underline{\tau}}$ having $m$ columns starting with $c_{(k-1)m+1}$ and ending with $c_{km}$. 
 
For $j$ odd,  let $l(j)$ be $\lfloor l_{i_j}/(m+1) \rfloor$. So $T^{l(j)}_{\underline{\tau}}$ is the subtableau containing the last occurrence of $i_j$ in the bottom i.e containing $C_{l_{i_j}}$ as one of its $m$ columns. For $j$ even let 
$f(j)$ be $\lfloor f_{i_j}/(m+1) \rfloor$. So $T^{f(j)}_{\underline{\tau}}$ is the subtableau containing the first occurrence of $i_j$ in the bottom row i.e containing $C_{f_{i_j}}$ as one of its $m$ columns.
We denote the first column of $T^k_{\underline{\tau}}$ by $T^k_{\underline{\tau}}[1]$ and the last column as $T^k_{\underline{\tau}}[m]$. 

For $j$ odd, let $S_{\underline{\tau}, j}$ denote the subtableau with columns $T_{\underline{\tau}}^{l(j)}[1], T_{\underline{\tau}}^{l(j)}[m],
T_{\underline{\tau}}^{l(j)+1}[1], T_{\underline{\tau}}^{l(j)+1}[m]$, $\ldots,$ $T_{\underline{\tau}}^{f(j+1)}[1] C_{f_{i_{j+1}}}$. 
Note that this tableau contains an even number of columns since $T_{\underline{\tau}}^{f(j+1)}[1]$ is different from $C_{f_{i_{j+1}}}$ - by definition $i_{j+1}$ appears only once in $T_{\underline{\mu}}$  and so its first occurence cannot be in a column numbered $1 \pmod m$ in $T_{\underline{\tau}}$ by Observation \ref{obs:modm}.

\eat{\begin{eqnarray}\nonumber
\beta_1(S_j) & = & C_{l_{i_j}}V(\Lambda_{\gamma_{i_j}}) \cr
\beta_2(S_j) &= & U(\Lambda_{\gamma_{i_j}+1})V(\Lambda_{\gamma_{i_j}+1}) \cr 
&\vdots & \cr
\beta_{t_{S_j}}(S_j) &= & U(\Lambda_{\gamma_{i_j}+t_{S_j}-1})C_{f_{i_{j+1}}} 
\end{eqnarray}
}

We denote by $S_{\underline{\tau},j}[k]$ the $2 \times 2$ subtableau of $S_{\underline{\tau},j}$ 
containing columns $2k-1$ and $2k$. To simplify notation we mostly omit the $\underline{\tau}$ and just denote this by $S_{j}[k]$ when $\underline{\tau}$ is clear from the context. 

Let $S_{\underline{\tau},j}[k]=
\begin{tabular}{|c|c|}\hline
 p & q  \\ \hline
 r  & s   \\ \hline
 \end{tabular} $. 

We set  $S_{\underline{\tau},j}[k](1)=p$,   $S_{\underline{\tau},j}[k](2)=q$,  $S_{\underline{\tau},j}[k](3)=r$ and  $S_{\underline{\tau},j}[k](4)=s$. 
 
\end{notation}

\eat{Using the following zero weight vector in $G_{2,7}$ we explain the notation set up above.}

\eat{\begin{definition}
For a standard monomial $F = p_{\tau_1}p_{\tau_2} \cdots p_{\tau_{mn}} $ define $N_F = \sum l(\tau_i)N^{(m+1)n-i-1}$ where $N >> 2(n-2)$. 
\end{definition}
}
We will use the degree lexicographic order on rectangular $2 \times m$, semistandard Young tableau. Recall that as per this order a monomial $p=p_{\tau_1}\ldots p_{\tau_m} $ corresponding to a rectangular $2 \times m$ is bigger than a monomial $q=q_{\mu_1} \ldots q_{\mu_{m'}} $ corresponding to $2 \times m'$ semistandard  
Young tableau if $m >m'$ or, if $m=m'$, then for the smallest $i$ such that $\tau_i \neq \mu_i$ it is the case that $\tau_i > \mu_i$ in the usual lexicographic order on $r$ length words.

Now we fix a $j$ which is odd and look the subtableau $S_j$ defined above for this $j$. Suppose $S_j$ has $2t$ columns.

\begin{lemma}
\label{lm:adj} 
For $1 \leq k < t$ we have $S_j[k](4) = S_j[k+1](3) $.  For $k$ such that $i_j \leq S_j[k](1) \leq i_{j+1}$ it is the case that $S_j[k](2)=S_j[k+1](1)$.  
\end{lemma}

\begin{proof}
If not, let $ S_j[k](4) \neq S_j[k+1](3)$ for some $k$.  Then $f_{S_j[k+1](3)} \equiv 1 \pmod m$. So we have $f^{S_j[k+1](3)} \equiv 1 \pmod m$ from Observation \ref{obs:modm}. If the number of times $S_j[k+1](3)$ appears in row $1$ or row 2 is not $m$ then $S_j[k+1](3)$ would occur 3 times in $T_{\underline{\mu}}$, a contradiction to the fact that $S_j[k+1](3)$ is not defected. So $S_j[k+1](3)$ appears $m$ times in row 1 and m times in row 2, and this pattern continues - all the intermediate $S_j[k+1](3)$, till we see $i_{j+1}$ appear $m$ times in the top and bottom row and occur first in the top and in the bottom in columns numbered $1 \pmod m$. But this force $f_{i_{j+1}} \equiv 1 \pmod m$, and as argued above $f^{i_{j+1}}\equiv 1 \pmod m$ - i.e. since $i_{j+1}$ is defected it has to occur 3 times which is a contradiction to lemma \ref{lm:alternate} since $j$ is odd.

The second statement has a similar proof and is omitted.
\end{proof}

\begin{lemma} 
\label{betanormal}
Let $j$ be odd and suppose $S_j$, has $2t_j$ columns for some $t_j$. Then for $1 \leq k \leq t_j$ we have $S_j[k](3) > S_j[k](2)$.
\end{lemma}

\begin{proof}
We prove it for $j=1$. We first show this for $k=1$.  
Assume $S_1[1](3) < S_1[1](2)$. Consider the tableau $T_{\underline{\mu}}$. The column $[S_1[1](1), S_1[1](3)]$ is a column numbered $1 \pmod m$ in $T_{\underline{\tau}}$ and so this column appears in 
$T_{\underline{\mu}}$.  If $S_1[1](3) < S_1[1](2)$, then all occurrences of $S_1[1](3)$ in
$T_{\underline{\mu}}$ appear in this column and to the left. The total number of positions in the boxes to the left of this column (including this column) in $T_{\underline{\mu}}$ is an even number. But $S_j[1](3)$ appears 3 times in these boxes since $i_1$ has defect 3,  and each other integer appears an even number of times since they are not defected. This is a contradiction.

Now we show this for $k >1$. Note that the column $[S_1[k](1),S_1[k](3)]$ occurs in
$T_{\underline{\mu}}$ since it is a column numbered $1 \pmod m$ in $T_{\underline{\tau}}$. If  $S_1[k](3) > S_1[k](2)$ then all occurrences of $S_1[k](3)$ in $T_{\underline{\mu}}$
are in this column and to its left. This is true for $S_1[k](1)$ too. Since $S_1[k](1)$ and $S_1[k](3)$ are not defected, they appear twice. The total number of positions to the left of (and including this ) column  $[S_1[k](1),S_1[k](3)]$ in  $T_{\underline{\mu}}$ is even. As before this is a contradiction since $i_1$ appears 3 times and all the other numbers appear twice. 

For $j$ odd and bigger than $1$, the proof is similar. 
Recall that the first column of $S_{\underline{\tau}, j}$ is column $T_{\underline{\tau}}^{l(j)}[1]$ and this appeares in $T_{\underline{\mu}}$.
The only point to note is that in 
$T_{\underline{\mu}}$, the columns strictly  to left of the column $T_{\underline{\tau}}^{l(j)}[1]$ contains all occurrences of the previous $i_k$, $k < j$ and the sum of the occurrences of
these $i_k$, $k < j$ is even. So too is the sum of occurrences of the remaining integers since they are not defected. The argument then proceeds as in the $j=1$ case.
\end{proof}

\begin{proposition} 
The map $R(1)^{\otimes m} \rightarrow R(m)$ is surjective.
\end{proposition}

\begin{proof}
\label{R_1}

The proof will be induction. For $p_{\underline{\tau}}$ in $R(m)$, we will show that there exists $p_{\underline{\mu}} \in R(1)$ and $p_{\underline{\nu}} \in R(m-1)$ and $p_{\underline{\tau}^j} \in R(m)$ such that $p_{\underline{\tau}} = p_{\underline{\nu}} p_{\underline{\gamma}} +  \sum_{j} p_{\underline{\tau}^j}  $ with  $p_{\underline{\tau}^j}  < p_{\underline{\tau}}$ in lexicographic order. Then an induction based on degree lexicographic order on monomials completes the proof.

The base case - the least monomial in lexicographic order is $p_{\underline{\tau}}$ corresponding to the semistandard Young tableau filled with $[1,2]$ in the first $2m$ columns and then $[3,4]$ and so on.  If we take the columns $1+j,m+1+j,2m+1+j,\ldots,(n-1)m+1+j$ for $0 \leq j \leq m-1$, the tableaus obtained are semistandard and the associated monomial is a zero weight vector $p_{\underline{\tau^j}}\in R(1)$. The product of these monomials is $p_{\underline{\tau}}$.

In general starting with $p_{\underline{\tau}}$ we construct $p_{\underline{\mu}}$ and 
$p_{\underline{\nu}}$ as given before Definition \ref{def:defect}, by taking for $T_{\underline{\mu}}$ the subtableau with columns $1,m+1,\ldots,(n-1)m+1$. If $p_{\underline{\mu}}$ is a zero weight (i.e in the corresponding tableau no integer is defected) we are done.  $p_{\underline{\tau}}$ is the
product of a zero weight vector in $R(1)$ and an element in $R(m-1)$ and we are done by induction on degree. 

Otherwise, proceeding as above we have defected integers $\{i_1,i_2,\ldots, i_{2l}\}$.  Corresponding to the integer $i_j$ in $\{i_1,i_3,\ldots,i_{2l-1}\}$, we have subtableaus $S_j$ and lemma \ref{betanormal} holds. For $j$ odd let the number of columns in $S_j$ be $2 t_j$.

Case 1 : Suppose for all $j$ and $1 \leq k \leq t_j$ it is the case that $S_j[k](3) > S_j[k](2) $. 

In this case we do the following operation : We change $S_j[k](3)$ to $S_j[k](4)$ and keep $ S_j[k](1)$, $ S_j[k](2) $ fixed for all $j$ odd and for all $1 \leq k \leq t_j$. We get a new Young tableau call it $S_j'$. We modify the original tableau corresponding to $p_{\underline{\tau}}$ by replacing the columns which were previously used to get $S_j$ by the corresponding columns of 
$S_j'$. We do this for every $j$. 

Denote the new monomial computed by this tableau by $p_{\underline{\tau'}}$ and denote by $p_{\underline{\mu'}}$ the monomial obtained from this tableau by selecting columns numbered $1,m+1,2m+1,\ldots, (n-1)m+1$. It is clear that $T_{\underline{\mu'}}$ is semi-standard.  Furthermore for every $j$ odd, one of the $i_j$'s which appeared in a column numbered $1 \pmod \ $ in $T_{\underline{\mu}}$ appears in now in a column numbered $0 \pmod m$, and so it's count in $T_{\underline{\mu'}}$ is one less than in $T_{\underline{\mu}}$. So $i_j$ is not defected in $T_{\underline{\mu'}}$.  For this same $j$ the last exchange is done between $S_j[t_j](3)$ and $S_j[t_j](4)$ and this is $i_{j+1}$. So this $i_{j+1}$ now occurs in a column numbered $1 \pmod m$ in $T_{\underline{\tau'}}$, and so the count of $i_{j+1}$ in $T_{\underline{\mu'}}$ is one more than in $T_{\underline{\mu}}$. So $i_{j+1}$ is not defected in $T_{\underline{\mu'}}$. Since $S_j[k](4) = S_j[k+1](3)$ for all $k$, the counts of the remaning integers in $T_{\underline{\mu'}}$ is the same as their count in $T_{\underline{\mu}}$, so these remain not defected. This is true for every $j$. So no integer is defected in $T_{\underline{\mu'}}$ and the corresponding monomial is a zero weight vector in $R(1)$. So  $p_{\underline{\tau'}}$ is a product of a zero weight monomial in $R(1)$ and an element of $R(m-1)$.   increases the count of reduced the number of $i_j$ by one and increased the number of $i_{j+1}$ by one. 

To finish the proof in this case we compare $p_{\underline{\tau}}$ with $p_{\underline{\tau'}}$.
Let us deonote the set of columns of $T_{\underline{\tau}}$ not in any $S_j$ by $Q$ and the monomial computed by them as $y$. If $S_{\underline{\tau},j}$ has $2t_j$ columns the monomial computed by it is a product of the $t_j$ monomials computed by the $2 \times 2$ subtableaus $S_{\underline{\tau},j}[k]$,
 $p_{S_{j}[k]}=p_{(S_j[k](1),S_j[k](4))}p_{(S_j[k](2),S_j[k](3))}$.  
We have
\begin{eqnarray}
p_{\underline{\tau}}= y \cdot \Pi_{j=1}^{j=l} \Pi_{k=1}^{k=t_j}p_{S_{j}[k]} \label{eqn:tau1}\\
p_{\underline{\tau'}}= y \cdot \Pi_{j=1}^{j=l} \Pi_{k=1}^{k=t_j}p_{S'_{j}[k]} \label{eqn:tau2}
\end{eqnarray} 

From the straightening laws the following relation holds between the tableaus
$S_{j}[k]$ and $S'_{j}[k]$. 
\begin{eqnarray} 
\label{eqn:str}\begin{tabular}{|c|c|}\hline
 p & q  \\ \hline
 r  & s   \\ \hline
 \end{tabular}
=\begin{tabular}{|c|c|}\hline
 p & q  \\ \hline
 s  & r   \\ \hline
 \end{tabular}
\pm \begin{tabular}{|c|c|}\hline
 p & r  \\ \hline
 q & s   \\ \hline
 \end{tabular}
\end{eqnarray}
Recall that in the equation above $S_{j}[k]$ is the tableau on the left hand side of the equation and 
$S'_{j}[k]$ is the first tableau on the right side. 

Plugging this into Equation \ref{eqn:tau1} above
we see that $p_{\underline{\tau}}$ is the sum of $p_{\underline{\tau'}}$ and sums of products of monomials obtained from $p_{\underline{\tau}}$ by replacing at least one
of the terms $p_{S_{j}[k]}$ in its expression by $p'_{j}[k]$, the monomial computed by the second tableau on the right hand side of Equation \ref{eqn:str}.  However since $r > q$ (from Lemma \ref{betanormal}), it follows that the second tableau on the right is lexicographically smaller than the tableau  $S_{j}[k]$. So the $2 \times mn$ tableau corresponding to each additional term obtained by plugging Equation \ref{eqn:str} into Equation \ref{eqn:tau1} is lexicographically smaller than $T_{\underline{\tau}}$. It is possible that this tableau is not semistandard and needs to be straightened into a sum of semistandard tableaus. But each such tableau $T_{\underline{\tau''}}$, will be lexocographically smaller than the (non semistandard) tableau we started with.  We proved above that $p_{\tau'}$ is the product of $p_{\underline{\mu'}} \in R(1)$ and a monomial $p_{\underline{\nu'}}\in R(m-1)$.  We have
\[ p_{\underline{\tau}}= p_{\underline{\mu'}} p_{\underline{\nu'}} + \sum_s p_{\underline{\tau''_s}}\] 
the sum being over tableaus which are smaller than $\tau$ in lexicographic order. By induction on lexicographic order each of these is in the image of $R(1)^{\otimes m}$. By induction on degree $p_{\underline{\nu'}}$ is in the image of $R(1)^{\otimes (m-1)}$.
So we are done.

Case 2 : For $j$ in which the conditions of Case 1 hold we do exactly as in that case. Let $j$ be such that $S_j[k](3) = S_j[k](2)$ for some $1 \leq k < t_j$. For each such $j$ we do the following.  First note that for such a $j$,  $i_j$ appears in $S_j$ as $S_j[k](1)$ for some $1 < k < t_j$, since it has defect 3. Let $m$ be the set of all elements $i_j \leq m < i_{j+1}$ with $m=S_j[k](3)$ in some subtableau $S_j[k]$ with $S_j[k](3) = S_j[k](2)$.  Order this set as $\{m_1, m_2, \cdots, m_e \}$ such that $ i_{j} \leq m_1 < m_2 < m_3 \cdots < m_e < i_{j+1}$, and let $k_i$ denote the index for which $m_i=S_j[k_i](3)=S_j[k_i](2)$ - clearly $m_s \neq m_t$ for $s\neq t$ and $m_e < i_{j+1}$ and $k_i \geq k$.  Let $x_i, y_i$ denote $S_{j}[k_i](1)$ and $S_{j}[k_i](4)$. 
For $k < k_i < t_j$ it follows from Lemma \ref{lm:adj} that $S_j[k_i-1](4) = m_i$,
$S_j[k_i+1](1)=m_i$, $S_j[k_i-1](2) = x_i$ and $S_j[k_i+1](3)=y_i$.
We have two subcases.
\begin{itemize}
\item[i] $e$ is odd: In this case we first swap $S_j[l](1)$ and $S_j(l)(2)$ for all $k \leq l \leq k_1-1$.
Then swap the two columns in $S_{k_1}$. And swap $S_j[l](3)$ and $S_j[l](4)$ for all 
$k_1+1 \leq l \leq k_2-1$.  Do nothing with $S_j[k_2]$. Instead start with $m_2$ which appears in $S_{j}[k_2+1](1)$ and repeat these steps. Since $m_e$ is odd, the last set of swap will happen in the bottom row starting from $y_{e} = S_j[k_e+1](3)$ up to $i_{j+1}=S_{j}[t_j](4)$. 
\item[ii] $e$ is even: In this case we swap $S_j[l](3)$ and $S_j(l)(4)$ for all $1 \leq l \leq k_1-1$.
Do nothing with $S_j[k_1]$. Instead swap $S_j[l](1)$ and $S_j[l](2)$ for all 
$k_1+1 \leq l \leq k_2-1$ and then swap the two columns of $S_j[k_2]$. And repeat the procedure from the $y_2$ which appears as $S_{k_2+1}[l](3)$. Since $m_e$ is even it can be checked that the last swaps will happen in the bottom row from $y_{e}=S_j[k_e+1](3)$ to $i_{j+1}=S_j[t_j](4)$.
\end{itemize}

After these round of swaps, we can use straightening as we did in case 1 above, to complete the proof.  The last set of swaps take place in the bottom row starting with an element occurring in a $1 \pmod m$ position and ending with the first occurence of $i_{j+1}$ in the bottom row - this is true in both cases. In both cases the first set of swaps start with $i_j$ occurring in a $1 \pmod m$ position and end with an element 
occurring in a position $0 \pmod m$.  It can be checked that if we form tableau $S_{j'}$ as we did in Case 1 above, the number of $i_j$ has reduced and the number of $i_{j+1}$ has increased. The number of occurrences of the intermediate numbers does not change because of the column swaps performed.  Furthermore the other set of swaps 
between elements in the top row and elements in the bottom row in an $S_j[l]$ take place in those $l$ wherein $S_j[l](3) > S_j[l](2)$. One checks as in Case 1 above that 
straightening introduces new zero weight tableaus, but all of them are lexicographically smaller than the tableau we start with. This completes the proof.

\end{proof}

\begin{theorem} 
$(\gmnmodt{2}{n})_T^{ss}({\cal L}(n\omega_2)), \tilde{{\cal L}}(n \omega_2))$ is projectively normal.
\end{theorem}

\begin{proof}
Now $$(\gmnmodt{2}{n})_T^{ss}({\cal L}(n\omega_2)), \tilde{{\cal L}}(n \omega_2))$$ is normal.  From Proposition \ref{R_1} we have the $R_1$ generation. The theorem follows.
\end{proof}

\begin{corollary}
The GIT quotient of a Schubert variety in $G_{2,n}$ is projectively normal with respect to the descent line bundle.
\end{corollary}
\begin{proof}
$T$ is a linearly reductive group.  For a Schubert variety $X(w)$ in $G(2,n)$ the map 
$H^0(G_{2,n},\mathcal{L}(n\omega_2)^{\otimes m})^T \longrightarrow H^0(X(w),\mathcal{L}(n\omega_2)^{\otimes m})^T$ is surjective. 
Since $X(w)_{T}^{ss}(\mathcal{L}(n\omega_2))$ is normal the corollary follows.
\end{proof}

We have an analogue of Corollary \ref{rem:prj}. The proof is similar and is omitted.
\begin{corollary}
The GIT quotient of a Richardson varieties in $G_{2,n}$ is projectively normal with respect to the descent line bundle.
\end{corollary}

\bibliography{KKS}
\bibliographystyle{ieeetr}
\end{document}